\documentclass[reqno,11pt,a4paper]{amsart}

\usepackage{amsfonts,amsmath,amssymb,amsthm,color}
\usepackage[hmargin=2cm,vmargin=2.5cm]{geometry}
\usepackage[colorlinks=true]{hyperref}
\hypersetup{colorlinks,citecolor=red,linkcolor=blue,urlcolor=black}

\newcommand\e\varepsilon
\newcommand\C{\mathbb C}
\newcommand\R{\mathbb R}

\newcommand\D{\mathbb D}
\renewcommand{\(}{\left(}
\renewcommand{\)}{\right)}

\newtheorem{theorem}{Theorem}[section]
\newtheorem{lemma}[theorem]{Lemma}

\newtheorem{proposition}[theorem]{Proposition}

\newtheorem{corollary}[theorem]{Corollary}

\allowdisplaybreaks

\everymath{\displaystyle}

\begin{document}

\title[Large conformal metrics with prescribed Gaussian and geodesic curvatures]{Large conformal metrics with prescribed Gaussian and geodesic curvatures}

\author{Luca Battaglia}
\address[L. Battaglia]{Dipartimento di Matematica e Fisica, Università degli Studi Roma Tre, Largo S. Leonardo Murialdo, 00146 Roma, Italy}
\email{lbattaglia@mat.uniroma3.it}

\author{Mar\'ia Medina}
\address[Mar\'ia Medina]{Departamento de Matem\'aticas,
Universidad Aut\'onoma de Madrid,
Ciudad Universitaria de Cantoblanco,
28049 Madrid, Spain}
\email{maria.medina@uam.es}

\author{Angela Pistoia}
\address[A. Pistoia]{Dipartimento SBAI, ``Sapienza" Universit\`a di Roma, via Antonio Scarpa 16, 00161 Roma, Italy}
\email{angela.pistoia@uniroma1.it}

\begin{abstract}
We consider the problem of prescribing Gaussian and geodesic curvatures for a conformal metric on the unit disk. This is equivalent to solving the following P.D.E.
$$\begin{cases}-\Delta u=2K(z)e^u&\hbox{in}\;\D^2,\\
\partial_\nu u+2=2h(z)e^\frac u2&\hbox{on}\;\partial\D^2,\end{cases}$$
where $K,h$ are the prescribed curvatures.
We construct a family of conformal metrics with curvatures $K_\e,h_\e$ converging to $K,h$ respectively as $\e$ goes to $0$, which blows up at one boundary point 
under some generic assumptions.
\end{abstract}

\date\today
\thanks{M. Medina was supported by the European Union's Horizon 2020 research and innovation programme under the Marie Sklodowska-Curie grant agreement N 754446 and UGR Research and Knowledge Transfer Fund - Athenea3i. 
A. Pistoia was partially supported by Fondi di Ateneo ``Sapienza" Universit\`a di Roma (Italy).
M. Medina wants to acknowledge the hospitality of Universit\`a La Sapienza di Roma, where this work was carried out during a long visit in the academic year 2019--2020.}

\maketitle

\section{Introduction}

Given a compact Riemannian surface $(\Sigma,g)$, a classical problem in Riemannian geometry is the prescription of the Gaussian curvature on $\Sigma$ under a conformal change of metric, dating back to Berger \cite{be} and Kazdan and Warner \cite{kw}. This geometric problem can be rephrased into a PDE. Indeed if $K_g$ is the Gaussian curvature relative to the metric $g$ and $K$ is the prescribed curvature, we can write the new metric conformal to $g$ as $\tilde g=e^ug$, and $K$ is the Gaussian curvature relative to $\tilde g$ if $u$ solves the problem
\begin{equation}\label{pres}
-\Delta_g u+2 K_g(x)=2 Ke^u\;\hbox{in}\;\Sigma,
\end{equation}
where $\Delta_g$ is the Laplace-Beltrami operator associated to the metric $g$.
When $\Sigma$ is the standard sphere $\mathbb S^2$, problem \eqref{pres} is called the {\em Nirenberg problem} and is specially interesting from the geometrical point of view due to the effect of the noncompact group of conformal maps. The solvability of problem \eqref{pres} has been studied for several decades. We refer the interested reader to Chapter 6 in the book \cite{a}, where we can also find a comprehensive list of references.\\

If $\Sigma$ has a boundary, other than the Gaussian curvature in $\Sigma$ it is natural to prescribe also the geodesic curvature on the boundary $\partial\Sigma$. This geometric problem can be rephrased into a PDE too. Using the same notation as above, if $h_g$ is the geodesic curvature of the boundary relative to the metric $g$ and $h$ is the prescribed curvature, $h$ is the geodesic curvature relative to $\tilde g$ if $u$ solves the boundary value problem
\begin{equation}\label{p0}
\begin{cases}-\Delta_g u+2 K_g(x)=2 K(x)e^u&\hbox{in}\;\Sigma,\\
\partial_\nu u+2h_g(x)=2 h(x) e^\frac u2&\hbox{on}\;\partial\Sigma.
\end{cases}
\end{equation}
The literature about problem \eqref{p0} is not as wide as the one concerning \eqref{pres}. A natural obstruction to the existence of solutions is given by the identity, obtained integrating \eqref{p0} and applying the Gauss-Bonnet theorem,
\begin{equation}\label{gb}
\int\limits_\Sigma Ke^u+\oint\limits_{\partial\Sigma}he^\frac u2=2\pi\chi(\Sigma),
\end{equation}
where $\chi(\Sigma)$ is the Euler characteristic of $\Sigma$. When $K$ and $h$ are constants Brendle obtained in \cite{bre} solutions to \eqref{p0} using a parabolic flow. Zhang in \cite{zhang} and Li and Zhu in \cite{lz} obtained some classification results for the solutions to \eqref{p0} in the case of the half-plane (see also G\'alvez and Mira \cite{gm}). The case of nonconstant curvatures was addressed for the first time by Cherrier in \cite{che}, where the existence of a solution to \eqref{p0} is proved provided the curvatures are not too big in a reasonable geometric sense. Recently, L\'opez-Soriano, Malchiodi and Ruiz in \cite{lmr} considered surfaces with negative Euler characteristic and negative Gaussian curvature, studied the problem via a variational point of view and obtained solutions to \eqref{p0} by minimization and min-max techniques.

When $\Sigma$ is the standard disk $\D^2$ problem \eqref{p0} can be seen as a generalization of the Nirenberg problem to surfaces with boundary. A natural obstruction to the existence of solution is \eqref{gb} with $\chi\(\D^2\)=1$, which implies that $K$ or $h$ must be positive somewhere. Hamza in \cite{h} found some integrability conditions, analogous to those of Kazdan and Warner for the sphere, that are necessary for solving the problem. The case $h=0$ was firstly studied by Chang and Yang in \cite{chyg}, where the authors proved the existence of a solution to \eqref{p0} provided $K$ is positive somewhere. The case $K=0$ was firstly considered by Chang and Liu in \cite{cl}, where the authors found a solution to \eqref{p0} under suitable assumptions of the critical points of geodesic curvature $h$. Successively, Chang, Xu and Yang in \cite{cxy} proved the existence of a solution when the geodesic curvature is close enough to the constant, using a perturbative approach, and Liu and Huang in \cite{lh} found a solution in a symmetric setting. As far as we know the only result for non-constant curvatures is due to Cruz and Ruiz in \cite{cr} where they found a solution via a variational argument under symmetry assumptions. Actually, a careful analysis of blow-up sequences is needed to find solutions to \eqref{p0} in the more general situation. If $K=0$, a blow-up analysis has been performed by Guo and Liu in \cite{gl}. Recently, Jevnikar, L\'opez-Soriano, Medina and Ruiz in \cite{jlsmr} performed a complete blow-up analysis of problem \eqref{p0} for non-constant $K$ and $h$. In particular, they consider a sequence of solutions $u_n$ to the problems
$$\begin{cases}-\Delta u_n=2K_n(z) e^{u_n}&\hbox{in}\;\D^2,\\
\partial_\nu u_n+2=2h_n(z) e^\frac {u_n}2&\hbox{on}\;\partial\D^2,\end{cases}$$
with $K_n\to K$ in $C^2\(\overline{\D^2}\)$ and $h_n\to h$ in $C^2\(\partial\D^2\)$ with bounded mass, i.e.
$$\int\limits_{\D^2}e^{u_n}+\int\limits_{\partial\D^2}e^\frac{u_n}2\le C$$
and they prove that if $u_n$ blows-up, i.e. $\sup_{\D^2}u_n\to+\infty$, then  the blow-up can only occur at a unique point $\xi_0$ of the boundary such that $h(\xi_0)^2+K(\xi_0)>0$ and $h(\xi_0)+\sqrt{h(\xi_0)^2+K(\xi_0)}>0$, which has to be a critical point of the map
\begin{equation}\label{phi}
\varphi(\xi):=H(\xi)+\sqrt{H(\xi)^2+K(\xi)}\quad\quad\quad\xi\in\D^2,
\end{equation}
where $H$ is the harmonic extension of $h$, that is,
$$\begin{cases}\Delta H=0&\hbox{in}\;\D^2,\\
H=h&\hbox{on}\;\partial\D^2.\end{cases}$$

Motivated by the previous result, we asked the following natural question: {\it if $\xi_0\in\partial\D^2$ is a critical point of the function $\varphi$, does there exist a family of solutions which blow-up at $\xi_0$?}\\

In the present paper we will give a positive answer and we will also find blowing-up solutions with a method that yields accurate estimates of their behavior.
 
In order to state our main result, let us introduce the assumptions.

Let $K\in C^2\(\overline{\D^2}\)$ and $h\in C^2\(\partial\D^2\)$. Without loss of generality we can assume $\xi_0=1\in\partial\D^2$, and we impose the necessary conditions found in \cite{jlsmr}: $\xi_0=1$ is a critical point of the function $\varphi$ defined in \eqref{phi} satisfying
\begin{equation}\label{h2k}
h(1)^2+K(1)>0,\quad\quad\quad\varphi(1)=h(1)+\sqrt{h(1)^2+K(1)}>0.
\end{equation}
 In particular, $\nabla\varphi(1)=0$ is equivalent to
\begin{equation}\label{critico}\partial_1K(1)+2\varphi(1)(-\Delta)^\frac12h(1)=0\quad\hbox{and}\quad\partial_2K(1)+2\varphi(1)h'(1)=0.\end{equation}
Here, by $\partial_1,\partial_2$ we denote the standard partial derivatives and a similar notation is used for second derivatives. We also denote by $(-\Delta)^\frac12h$ the fractional Laplacian of $h$ along the boundary of the disk, i.e.,
$$(-\Delta)^\frac12h(z)=\frac1\pi\mbox{p.v.}\int_{\partial\D^2}\frac{h(z)-h(w)}{|z-w|^2}dw,$$
that naturally arises by noticing that $\partial_\nu H|_{\partial\D^2}=(-\Delta)^\frac12h$, where $\nu$ is the unit outward normal at the boundary.\\

We require the following {\it non-degeneracy} conditions (see Corollary \ref{mainorder} and Section \ref{6}) on the critical point:
\begin{eqnarray}
\label{cond10}&\Delta K(1)+4|\nabla H(1)|^2\ne0,\\
\label{cond100}&\partial_{22}K(1)+2\varphi(1)h''(1)\ne0,
\end{eqnarray}
and we additionally assume
\begin{equation}\label{cond1}
2K(1)+\varphi(1)h(1)\ne0.
\end{equation}
Notice that the case $K=0$ does not satisfy assumption \eqref{cond10} (it would contradict \eqref{critico}). It will be treated in a forthcoming paper. \\

Let us now introduce suitable linear perturbations of the functions $K$ and $h$ which lead to the existence of blowing-up solutions. More precisely let us denote
$$K_\e(z):=K(z)+\e G(z)\quad\hbox{and}\quad h_\e(z):=h(z)+\e I(z),$$
with $G\in C^2\(\overline{\D^2}\),I\in C^2\(\partial\D^2\)$. We choose these perturbations vanishing at the concentration point, i.e.
\begin{equation}\label{gi1}
G(1)=I(1)=0,
\end{equation}
and satisfying the {\it generic} condition 
\begin{equation}\label{cond1000} 
\partial_1G(1)+2\varphi(1)(-\Delta)^\frac12I(1)\not=\(\partial_2G(1)+2\varphi(1)I'(1)\)\frac{\partial_{12}K(1)+2\varphi(1)(-\Delta)^\frac12h'(1)}{\partial_{22}K(1)+2\varphi(1)h''(1)}.
\end{equation}
Let us briefly comment these assumptions. Condition \eqref{gi1} is not restrictive and it is just taken for simplicity.
On the other hand, although \eqref{cond1000} may look involved, it is actually a rather natural transversality condition on the perturbations $G,I$. In fact, if we perturb the map $\varphi$ as
$$\widetilde\varphi(\xi,\e):=H_\e(\xi)+\sqrt{H_\e(\xi)^2+K_\e(\xi)},$$
with $H_\e$ being the harmonic extension of $h_\e$, a simple computation shows that \eqref{cond1000} is equivalent to assuming the vectors $\partial_{\xi_2}\nabla_\xi\widetilde\varphi(0,0),\partial_\e\nabla_\xi\widetilde\varphi(0,0)$ to be not parallel, namely
$$\det\(\begin{array}{cc}\partial_{\xi_1,\xi_2}\widetilde\varphi(0,0)&\partial_{\xi_1,\e}\widetilde\varphi(0,0)\\\partial_{\xi_2,\xi_2}\widetilde\varphi(0,0)&\partial_{\xi_2,\e}\widetilde\varphi(0,0)\end{array}\)\ne0.$$

Finally, let us state our main result.
 
\begin{theorem}\label{main}
Assume \eqref{h2k}-\eqref{cond1000}. 
Then, there exists $\e_0>0$ such that, for every $\e\in(0,\e_0)$ or for every $\e\in(-\e_0,0)$ there exists a solution $u_\e$ of 
\begin{equation}\label{p}
\begin{cases}-\Delta u=2K_\e(z) e^u&\hbox{in}\;\D^2,\\
\partial_\nu u+2=2h_\e(z)e^\frac u2&\hbox{on}\;\partial\D^2,\end{cases}
\end{equation}
which blows-up at $1$ as $\e\to0$.\\
Moreover, there exist $\delta_\e>0$ and $\xi_\e\in\partial\D^2$ with
$$\delta_\e=O\(\frac{|\e|}{\log\frac1{|\e|}}\)\quad\hbox{and}\quad\xi_\e=1+O(|\e|)$$
such that
$$u_\e\(f_{\delta_\e,\xi_\e}(z)\)+2\log\left|f'_{\delta_\e,\xi_\e}(z)\right|-2\log\(\frac{2\varphi(1)}{\varphi(1)^2+K(1)|z|^2}\)=O\(\frac{|\e|}{\log\frac1{|\e|}}\)\quad\hbox{in}\quad H^1\(\D^2\).$$
Here $f_{\delta,\xi}$ is the conformal map 
\begin{equation}\label{confo}
f=f_{\delta,\xi}(z):=\frac{z+(1-\delta)\xi}{1+(1-\delta)\overline\xi z},\qquad \delta\in\R,\; z,\xi\in \C.
\end{equation}
\end{theorem}
${}$\\
From now on we will make the identification $\R^2\simeq \C$, and thus we will understand every point $z, \xi$ as points in the complex plane. Let us briefly sketch the idea of the proof and clarify the role of the assumptions. First of all we remark that $u$ is a solution of \eqref{p} if and only if the function
$$v(z):=u(f(z))+2\log|f'(z)|$$ solves 
\begin{equation}\label{probV}
\begin{cases}-\Delta v=2K_\e(f(z))e^v&\hbox{in}\;\D^2,\\
\partial_\nu v+2=2h_\e(f(z))e^\frac v2&\hbox{on}\;\partial\D^2.\end{cases}
\end{equation}
where $f=f_{\delta,\xi}$ is the conformal map defined in \eqref{confo}. We will choose $\delta=\delta_\e\underset{\e\to0}\to0$ and $\xi=\xi_\e\in\partial\D^2$ with $\xi_\e\underset{\e\to0}\to\xi_0=1$ (equivalently, $\xi=e^{\imath\eta}$ with $\eta=\eta_\e\underset{\e\to0}\to0$). 

To prove Theorem \ref{main} we will perform a Lyapunov-Schmidt type reduction to find a solution of \eqref{probV} which requires some delicate and quite technical adjustements. Roughly speaking, we will construct a solution of the form
$$v(z)=V (z)+o_\e(1),$$
where 
\begin{equation}\label{Vbox}
\boxed{ V=V_\xi(z):=2\log\(\frac{2\varphi(\xi)}{\varphi(\xi)^2+K(\xi)|z|^2}\),}
\end{equation}
with $\varphi$ defined in \eqref{phi}. $V$ satisfies 
\begin{equation}\label{V}
\begin{cases}-\Delta V=2K(\xi)e^V&\hbox{in}\;\D^2\\
\partial_\nu V+2=2h(\xi)e^\frac V2&\hbox{on}\;\partial\D^2.\end{cases}
\end{equation}
Taking $V$ as first approximation to carry on the reduction procedure is not enough since it produces a too large error. To overcome this difficulty we need to improve the ansatz by adding the term 
\begin{equation}\label{Wbox}
\boxed{W=W_\xi(z):=-\frac2\pi \frac{2\varphi(\xi)}{\varphi(\xi)^2+K(\xi)}\int\limits_{\partial\D^2}\log|z-w|(h(f(w))-h(\xi))dw,}\end{equation}
which solves the linear problem
\begin{equation}\label{eqw}
\begin{cases}-\Delta W=0&\hbox{in}\;\D^2,\\
\partial_\nu W=2(h(f(z))-h(\xi))e^\frac V2-\frac1\pi e^\frac{V}2\int\limits_{\partial\D^2}(h(f(w))-h(\xi))dw&\hbox{on}\;\partial\D^2.\end{cases}
\end{equation}
Notice that $V$ is constant along the boundary $\partial\D^2$. On the other hand, the solution to \eqref{eqw} is not unique but invariant under the addition of constants and so we need to refine again the ansatz by adding a small constant $\tau=\tau_\xi$. Finally, we will look for a solution of the form
$$v(z)=V(z)+W(z)+\tau+\phi(z),$$
where $\phi$ is a small term which has to be found via an accurate linear theory.\\
We point out that the presence of the extra parameter $\tau$ is not an innocent matter. Actually, it is necessary because we need to perform the linear theory in a suitable Hilbert space which does not contain the constants, which is possible thanks to assumption \eqref{cond1}. 
Once the error term $\phi$ is found, the existence of the solution $v$ with the prescribed profile is reduced to the existence of the three parameters $\delta$, $\xi$ and $\tau$, which is ensured by assumptions \eqref{cond10}, \eqref{cond100}, \eqref{gi1} and \eqref{cond1000}.\\
 
The paper is organized as follows: in Section 2 we detail the ansatz for the solution and precise estimates of the correction term $W$ and the associated error. We also tranform problem \eqref{probV} into a problem on $\phi$ whose main operator is the linearized, $\mathcal L_0$, associated to \eqref{V}. In Section 3 we develop the invertibility theory for $\mathcal L_0$ in the case of a linear problem. Section 4 is devoted to solve a non linear projected problem related to the problem that $\phi$ must satisfy. Section 5 contains the estimates on the projections of every term in the problem on the elements of the kernel, and Section 6 makes use of this information to perform the finite dimensional reduction, concluding the proof of Theorem \ref{main}. Finally, the appendices compile useful computations related to the estimates in Section 2 and Section 5.

\medskip

\section{Ansatz and error estimates}

We look for a solution of \eqref{probV} as
$$\boxed{v(z)=V(z)+W(z)+\tau+\phi(z).}$$
Here, $V=V_\xi$ and $W=W_\xi$ are given in \eqref{Vbox} and \eqref{Wbox} respectively, $\tau=\tau_\xi$ is a constant and $\phi(z)=\phi_{\xi,\delta,\tau}(z)$ is a smaller term. We want to find the three parameters $\delta=\delta_\e>0$ with $\delta_\e\underset{\e\to0}\to0$, $\xi=\xi_\e\in\partial\D^2$ with $\xi_\e\underset{\e\to0}\to1$ (or equivalently, $\xi=e^{\imath\eta}$ with $\eta=\eta_\e\underset{\e\to0}\to0$) and $\tau_\e\underset{\e\to0}\to0$ such that $\phi$ solves
$$\begin{cases}-\Delta\(V+W+\tau+\phi\)=2K_\e(f(z))e^{V+W+\tau+\phi}&\hbox{in}\;\D^2,\\
\partial_\nu(V+W+\tau+\phi)+2=2h_\e(f(z))e^\frac{V+W+\tau+\phi}2&\hbox{on}\;\partial\D^2,\end{cases}$$
where $f=f_{\delta,\xi}$ is the conformal map defined in \eqref{confo}. We require $\delta>0$ since, in this conformal map, we want $(1-\delta)\xi$ to be a sequence of points that approximates $\xi_0$ from the interior of the disk. The previous problem is equivalent to
\begin{eqnarray*}
-\Delta\phi-2K(\xi)e^V\phi&=&2\(K_\e(f(z))e^{W+\tau}-K(\xi)\)e^V\\
&+&2\(K_\e(f(z))e^{W+\tau}-K(\xi)\)e^V\phi\\
&+&2K_\e(f(z))e^{V+W+\tau}\(e^\phi-1-\phi\)\quad\hbox{in}\;\D^2,
\end{eqnarray*}
with boundary condition
\begin{eqnarray*}
\partial_\nu\phi-h(\xi)e^\frac V2\phi&=&2\(h_\e(f(z))e^\frac{W+\tau}2-h(f(z))\)e^\frac V2+\frac1\pi e^\frac{V}2\int\limits_{\partial\D^2}(h(f(w))-h(\xi))dw\\
&+&\(h_\e(f(z))e^\frac{W+\tau}2-h(\xi)\)e^\frac V2\phi\\
&+&2h_\e(f(z))e^\frac{V+W+\tau}2\(e^{\frac\phi2}-1-\frac\phi2\)\quad\hbox{on}\;\partial\D^2;
\end{eqnarray*}
which can be rewritten as
\begin{equation}\label{eqphi}
\begin{cases}\mathcal L^\mathrm{Int}_0\phi=\mathcal E^\mathrm{Int}+\mathcal L^\mathrm{Int}\phi+\mathcal N^\mathrm{Int}(\phi)&\hbox{in}\;\D^2,\\
\mathcal L^\partial_0\phi=\mathcal E^\partial+\mathcal L^\partial\phi+\mathcal N^\partial(\phi)&\hbox{on}\;\partial\D^2,\end{cases}
\end{equation}
with
\begin{eqnarray}
\nonumber\mathcal L^\mathrm{Int}_0\phi&:=&-\Delta\phi-2K(\xi)e^V\phi,\\
\label{eint}\mathcal E^\mathrm{Int}&:=&2\(K_\e(f(z))e^{W+\tau}-K(\xi)\)e^V,\\
\nonumber\mathcal L^\mathrm{Int}\phi&:=&2\(K_\e(f(z))e^{W+\tau}-K(\xi)\)e^V\phi,\\
\nonumber\mathcal N^\mathrm{Int}(\phi)&:=&2K_\e(f(z))e^{V+W+\tau}\(e^\phi-1-\phi\),
\end{eqnarray}
and, for $z\in \partial\D^2$,
\begin{eqnarray}
\nonumber\mathcal L^\partial_0\phi&:=&\partial_\nu\phi-h(\xi)e^\frac V2\phi,\\
\label{ebdr}\mathcal E^\partial&:=&2\(h_\e(f(z))e^\frac{W+\tau}2-h(f(z))\)e^\frac V2+\frac1\pi e^\frac{V}2\int\limits_{\partial\D^2}(h(f(w))-h(\xi))dw,\\
\nonumber\mathcal L^\partial\phi&:=&\(h_\e(f(z))e^\frac{W+\tau}2-h(\xi)\)e^\frac V2\phi,\\
\nonumber\mathcal N^\partial(\phi)&:=&2h_\e(f(z))e^\frac{V+W+\tau}2\(e^{\frac\phi2}-1-\frac\phi2\).
\end{eqnarray}
We require the parameters $\delta$, $\tau$ and $\eta$ (recall that $\xi=e^{\imath\eta}$) to satisfy
\begin{equation}\label{orders}
\boxed{\delta=O\(\frac{|\e|}{\log\frac1{|\e|}}\),\quad\tau=O\(\frac{|\e|}{\log\frac1{|\e|}}\),\quad\eta=O(|\e|).}
\end{equation}
This imposition on the orders is naturally justified at the final step of the finite dimensional reduction (see Section 6). We will assume them in advance in order to simplify the writing of the article. 
\\

\begin{lemma}\label{w}

The correction $W$ verifies
\begin{eqnarray*}
W(z)&=&\frac{8\varphi(\xi)}{\varphi(\xi)^2+K(\xi)}\delta\((-\Delta)^\frac12h(\xi)\log|z+\xi|+h'(\xi)\arctan\frac{\left\langle z,\xi^\perp\right\rangle}{1+\langle z,\xi\rangle}\)\\
&+&O\(\frac{\delta^2}{\delta+|z+\xi|}\(1+\left|\log\frac{|z+\xi|}\delta\right|\)\).
\end{eqnarray*}
In particular,
\begin{equation}\label{wi}W(z)=O\(\delta\(1+\log\frac1{|z+\xi|}\)\)\end{equation}
and  for any $p\in[1,+\infty)$ and $\delta$ sufficiently small one has
\begin{equation}\label{wp}
\|W\|_{L^p\(\D^2\)}+\|W\|_{L^p\(\partial\D^2\)}=O(\delta)\ \hbox{and}\
\left\|e^{|W|}\right\|_{L^p\(\D^2\)}+\left\|e^\frac{|W|}2\right\|_{L^p\(\partial\D^2\)}=O(1).
\end{equation}

\end{lemma}\

\begin{proof}
By Proposition \ref{f} we write
\begin{eqnarray*}
W(z)&=&-\frac2\pi\int\limits_{\partial\D^2}\log\left|z-w\right|(h(f(w))-h(\xi))\frac{2\varphi(\xi)}{\varphi(\xi)^2+K(\xi)}dw\\
&=&-\frac{4\varphi(\xi)}{\pi\(\varphi(\xi)^2+K(\xi)\)}\(\log|z+\xi|\underbrace{\int\limits_{\partial\D^2}(h(f(w))-h(\xi))dw}_{=:I_1}+\delta h'(\xi)\underbrace{\int\limits_{\partial\D^2}\log\frac{|z-w|}{|z+\xi|}\Theta(w)dw}_{=:I_2}\right.\\
&+&\left.\underbrace{\int\limits_{\partial\D^2}\log\frac{|z-w|}{|z+\xi|}\(h(f(w))-h(\xi)-\delta h'(\xi)\Theta(w)\)dw}_{=:I_3}\),
\end{eqnarray*}
where
$$\Theta(w):=\frac{2\left\langle w,\xi^\perp\right\rangle}{1+(1-\delta)^2+2(1-\delta)\langle w,\xi\rangle}.$$
From Proposition \ref{intFracLap} we know that
$$I_1=-2\pi\delta(-\Delta)^\frac12h(\xi)+O\(\delta^2\);$$
hence the result follows if we prove
\begin{eqnarray}
\label{ii2}I_2&=&-2\pi\arctan\frac{\left\langle z,\xi^\perp\right\rangle}{1+\langle z,\xi\rangle}+O\(\frac{\delta}{\delta+|z+\xi|}\),\\
\label{ii3}I_3&=&O(A(z)),
\end{eqnarray}
where 
$$A(z)=A_{\delta,\xi}(z):=\frac{\delta^2}{\delta+|z+\xi|}\(1+\left|\log\frac{|z+\xi|}\delta\right|\).$$
To deal with \eqref{ii2} we use the fact that $\Theta$ is odd with respect to the map $w\mapsto\xi^2\overline w$, so that
$$\int\limits_{\partial\D^2}\Theta(w)dw=0,$$
and, by Green's representation formula, $I_2$ is the solution to
$$\begin{cases}-\Delta I_2=0&\hbox{in}\;\D^2\\
\partial_\nu I_2=-\pi\Theta&\hbox{on}\;\partial\D^2,\end{cases}$$
that is
$$I_2=-\frac{2\pi}{1-\delta}\arctan\frac{(1-\delta)\left\langle z,\xi^\perp\right\rangle}{1+(1-\delta)\langle z,\xi\rangle}.$$
Using the fundamental theorem of calculus and estimate \eqref{deltazxi} we get
\begin{eqnarray*}
&&-\frac{2\pi}{1-\delta}\arctan\frac{(1-\delta)\left\langle z,\xi^\perp\right\rangle}{1+(1-\delta)\langle z,\xi\rangle}-\(-2\pi\arctan\frac{\left\langle z,\xi^\perp\right\rangle}{1+\langle z,\xi\rangle}\)\\
&=&2\pi\(\arctan\frac{\left\langle z,\xi^\perp\right\rangle}{1+\langle z,\xi\rangle}-\arctan\frac{(1-\delta)\left\langle z,\xi^\perp\right\rangle}{1+(1-\delta)\langle z,\xi\rangle}\)+2\pi\frac\delta{1-\delta}\arctan\frac{(1-\delta)\left\langle z,\xi^\perp\right\rangle}{1+(1-\delta)\langle z,\xi\rangle}\\
&=&-2\pi\int_0^{\delta}\frac{\left\langle z,\xi^\perp\right\rangle}{\left|1+(1-s)\overline\xi z\right|^2}ds+O(\delta)\\
&=&O\(\int_0^\delta\frac{|z+\xi|}{(s+|z+\xi|)^2}ds+\delta\)\\
&=&O\(\frac\delta{\delta+|z+\xi|}\)
\end{eqnarray*}
hence \eqref{ii2} is proved.

We are left with showing \eqref{ii3}; thanks to Proposition \ref{f}, we need to show:
$$\left|\int\limits_{\partial\D^2}\log\frac{|z-w|}{|z+\xi|}\(\frac\delta{\delta+|w+\xi|}\)^2dw\right|=O(A(z)).$$
We split the integral in three parts, depending on whether $z$ is much closer to $w$ than to $-\xi$, much farther, or the distances are comparable.

If $|z-w|\le\frac{|z+\xi|}2$, then $|w+\xi|\ge\frac{|z+\xi|}2$, and
\begin{eqnarray*}
\left|\int\limits_{\left\{|z-w|\le\frac{|z+\xi|}2\right\}}\log\frac{|z-w|}{|z+\xi|}\(\frac\delta{\delta+|w+\xi|}\)^2dw\right|
&=&O\(\(\frac\delta{\delta+\frac{|z+\xi|}2}\)^2\int\limits_{\left\{|z-w|\le\frac{|z+\xi|}2\right\}}\log\frac{|z+\xi|}{|z-w|}dw\)\\
&=&O\(\(\frac\delta{\delta+|z+\xi|}\)^2\)|z+\xi|\int\limits_{\left\{t<\frac12\right\}}\log\frac1tdt\\
&=&O\(\frac{\delta^2}{\delta+|z+\xi|}\)\\
&=&O(A(z)).
\end{eqnarray*}
If $\frac{|z+\xi|}2<|z-w|\le2|z+\xi|$, then
$$\left|\log\frac{|z-w|}{|z+\xi|}\right|=\left|\log\(1+\frac{|z-w|-|z+\xi|}{|z+\xi|}\)\right|=O\(\frac{|z-w|-|z+\xi|}{|z+\xi|}\)=O\(\frac{|w+\xi|}{|z+\xi|}\);$$
moreover, $|w+\xi|\le3|z+\xi|$, therefore
\begin{eqnarray*}
&&\int\limits_{\left\{\frac{|z+\xi|}2<|z-w|\le2|z+\xi|\right\}}\left|\log\frac{|z-w|}{|z+\xi|}\right|\(\frac\delta{\delta+|w+\xi|}\)^2dw\\
&=&O\(\int\limits_{\{|w+\xi|\le3|z+\xi|\}}\frac{|w+\xi|}{|z+\xi|}\(\frac\delta{\delta+|w+\xi|}\)^2dw\)\\
&=&O\(\int\limits_{\{|w+\xi|\le\min\{\delta,3|z+\xi|\}\}}\frac{|w+\xi|}{|z+\xi|}dw+\int\limits_{\{\min\{\delta,3|z+\xi|\}\le|w+\xi|\le3|z+\xi|\}}\frac{\delta^2}{|z+\xi||w+\xi|}dw\)\\
&=&O\(\frac{\min\{\delta,3|z+\xi|\}^2}{|z+\xi|}+\frac{\delta^2}{|z+\xi|}\log^+\frac{3|z+\xi|}\delta\)\\
&=&O(A(z)),
\end{eqnarray*}
where $\log^+$ means the positive part, that is, $\log^+\frac{3|z+\xi|}\delta:=\max\left\{0, \frac{3|z+\xi|}\delta\right\}$.
Finally, if $\frac{|z-w|}{|z+\xi|}>2$, then
$$\log\frac{|z-w|}{|z+\xi|}=O\(1+\log\(\frac{|z-w|}{|z+\xi|}-1\)\)=O\(1+\log\frac{|w+\xi|}{|z+\xi|}\);$$
moreover, $|w+\xi|\ge|z+\xi|$, therefore
\begin{eqnarray*}
&&\int\limits_{\left\{|z-w|>2|z+\xi|\right\}}\log\frac{|z-w|}{|z+\xi|}\(\frac\delta{\delta+|w+\xi|}\)^2dw\\
&=&O\(\int\limits_{\{|w+\xi|>|z+\xi|\}}\(1+\log\frac{|w+\xi|}{|z+\xi|}\)\(\frac\delta{\delta+|w+\xi|}\)^2dw\)\\
&=&O\(\int\limits_{\{|z+\xi|<|w+\xi|\le\delta\}}\(1+\log\frac{|w+\xi|}{|z+\xi|}\)dw+\int\limits_{\{|w+\xi|>\delta\}}\(1+\log\frac{|w+\xi|}{|z+\xi|}\)\frac{\delta^2}{|w+\xi|^2}dw\)\\
&=&O\(\int\limits_{\left\{1<t\le\frac\delta{|z+\xi|}\right\}}|z+\xi|(1+\log t)dt+\int\limits_{\left\{t>\frac\delta{|z+\xi|}\right\}}\frac{\delta^2}{|z+\xi|}\frac{1+\log t}{t^2}dt\)\\
&=&O\(\delta\log\frac\delta{|z+\xi|}+\frac{\delta^2}{|z+\xi|}\)\\
&=&O(A(z)).
\end{eqnarray*}

The pointwise estimate \eqref{wi} and the $L^p$-estimates \eqref{wp} follow immediately.
\end{proof}\

\begin{proposition}\label{e}
Consider $\mathcal E^\mathrm{Int},\mathcal E^\partial$ the interior and boundary error terms defined in \eqref{eint}, \eqref{ebdr} respectively. Then, for $1<p<2$,
$$\|\mathcal E\|_p:=\left\|\mathcal E^\mathrm{Int}\right\|_{L^p\(\D^2\)}+\left\|\mathcal E^\partial\right\|_{L^p\(\partial\D^2\)}=O\(\frac{|\e|}{\log\frac1{|\e|}}\).$$
\end{proposition}

\begin{proof}
Let us start with the estimate in $\D^2$. We split
\begin{eqnarray}
\nonumber\mathcal E^\mathrm{Int}&=&\underbrace{2(K(f(z))+\e G(f(z)))\(e^{W+\tau}-1\)e^V}_{=:\mathcal E^\mathrm{Int}_1}\\
\label{split1}&+&\underbrace{2(K(f(z))-K(\xi)+\e(G(f(z))-G(\xi)))e^V}_{=:\mathcal E^\mathrm{Int}_2}+\underbrace{2\e G(\xi)e^V}_{=:\mathcal E^\mathrm{Int}_3}.
\end{eqnarray}
To estimate the first term, we use the basic estimate
$$e^{{W+\tau}}-1=O\(\(1+e^{{W+\tau}}\)(W+\tau)\)$$
and   \eqref{wp}, which gives
$$\left\|\mathcal E^\mathrm{Int}_1\right\|_{L^p\(\D^2\)}=O\(\|K+\e G\|_{L^\infty\(\D^2\)}\left\|1+e^{W+\tau}\right\|_{L^{2p}\(\D^2\)}\|W+\tau\|_{L^{2p}\(\D^2\)}\left\|e^V\right\|_{L^\infty\(\D^2\)}\)=O(\delta+|\tau|).$$
For the second term, the mean value Theorem and Proposition \ref{f} give
\begin{eqnarray*}
\left\|\mathcal E^\mathrm{Int}_2\right\|_{L^p\(\D^2\)}
&=&O\(\|\nabla(K+\e G)\|_{L^\infty\(\D^2\)}\left\|f(z)-\xi\right\|_{L^p\(\D^2\)}\left\|e^V\right\|_{L^\infty\(\D^2\)}\)=O(\delta).
\end{eqnarray*}
The third term can be estimated using \eqref{gi1}, which gives $G(\xi)=O(|\eta|)$, hence
$$\left\|\mathcal E^\mathrm{Int}_3\right\|_{L^p\(\D^2\)}=O(|\eta||\e|).$$
Therefore
\begin{equation}\label{eintlp}
\left\|\mathcal E^\mathrm{Int}\right\|_{L^p\(\D^2\)}=O(\delta+|\eta||\e|+|\tau|).
\end{equation}
To estimate the boundary term we split it as
\begin{equation}\label{split2}
\!\!\!\!\!\!\!\!\!\!\!\!\!\!\!\!\!\!\!\!\!\!\!\!\!\!\!\!\!\!\mathcal E^\partial=\underbrace{2h(f(z))\(e^\frac{W+\tau}2-1\)e^\frac V2}_{=:\mathcal E^\partial_1}+\underbrace{2\e I(f(z))e^\frac{V+W+\tau}2}_{=:\mathcal E^\partial_2}+\underbrace{\frac1\pi e^\frac{V}2\int\limits_{\partial\D^2}(h(f(w))-h(\xi))dw}_{=:\mathcal E^\partial_3}.
\end{equation}
The first term can be estimated, similarly as before, using   \eqref{wp}:
\begin{eqnarray}
\nonumber\left\|\mathcal E^\partial_1\right\|_{L^p\(\partial\D^2\)}&=&O\(\|h\|_{L^\infty\(\partial\D^2\)}\left\|1+e^{W+\tau}\right\|_{L^{2p}\(\partial\D^2\)}\|W+\tau\|_{L^{2p}\(\partial\D^2\)}\left\|e^V\right\|_{L^\infty\(\partial\D^2\)}\)\\
\label{ebdr1lp}&=&O(\delta+\tau).
\end{eqnarray}
For the second boundary term we recall that, in view of \eqref{gi1}, $I(f(z))=O(|f(z)-\xi|+|\eta|)$, and therefore using Proposition \ref{f} and \eqref{wp} we get:
\begin{eqnarray}
\nonumber\left\|\mathcal E^\partial_2\right\|_{L^p\(\partial\D^2\)}&=&O\(|\e|\||f(z)-\xi|+|\eta|\|_{L^{2p}\(\partial\D^2\)}\left\|e^\frac{V+W+\tau}2\right\|_{L^{2p}\(\partial\D^2\)}\)\\
\nonumber&=&O\(|\e|\(\delta^\frac1{2p}+|\eta|\)\)\\
\label{ebdr2lp}&=&O\(|\e|\(\delta^\frac14+|\eta|\)\).
\end{eqnarray}
Finally, since $e^V$ is constant on $\partial\D^2$, we can exploit Proposition \ref{intFracLap} to get
$$\left\|\mathcal E^\partial_3\right\|_{L^p\(\partial\D^2\)}=O(\delta),$$
which concludes the proof.
\end{proof}

\medskip

\section{The linear theory}

Define
$$\mathfrak C:=\left\{\xi\in\partial\D^2:\;2K(\xi)\int\limits_{\D^2}e^{V_\xi}+ h(\xi)\int_{\partial\D^2}e^\frac{V_\xi}2=2K(\xi)\int\limits_{\D^2}e^{V_\xi}+h(\xi)\frac{4\pi\varphi(\xi)}{\varphi(\xi)^2+K(\xi)}\ne0\right\}.$$
For any $\xi\in\mathfrak C$ we consider the Hilbert space
$$\mathbf H_\xi:=\left\{\phi\in H^1\(\D^2\):\;2K(\xi)\int\limits_{\D^2}e^{V_\xi}\phi+h(\xi)\frac{2\varphi(\xi)}{\varphi(\xi)^2+K(\xi)}\int\limits_{\partial\D^2}\phi=0\right\},$$
equipped with the scalar product and the corresponding norm (see the following lemma)
$$\langle u,v\rangle:=\int\limits_{\D^2}\nabla u\nabla v\quad\hbox{and}\quad\|u\|:=\|\nabla u\|_{L^2\(\D^2\)}=\(\int\limits_{\D^2}|\nabla u|^2\)^\frac12.$$
We point out that non-zero constant functions do not belong to the space $\mathbf H_\xi$.

Notice that, in view of Proposition \ref{estZ01}, $\xi$ belongs to $\mathfrak C$ if and only if $2K(\xi)+\varphi(\xi)h(\xi)\ne0$; therefore, since we are assuming \eqref{cond1}, we have that $\xi\in\mathfrak C$ for any $\xi$ close enough to $1$.

\begin{lemma}\label{emb}\
\begin{enumerate}
\item For any compact set $\mathfrak C'\Subset\mathfrak C$ there exists a constant $C>0$ such that
$$\|\phi\|_{L^2\(\D^2\)}\le C\|\nabla\phi\|_{L^2\(\D^2\)}\quad\quad\quad\forall\phi\in\mathbf H_\xi,\;\xi\in\mathfrak C'.$$
\item The norm $\|\cdot\|$ is equivalent to the standard norm $\|\cdot\|_{H^1\(\D^2\)}$ and the embeddings
$$\mathbf H_\xi\hookrightarrow L^p\(\D^2\)\quad\hbox{and}\quad\mathbf H_\xi\hookrightarrow L^p\(\partial\D^2\)$$ are compact and continuous for any $p>1$
\item There exists a constant $C$ (depending only on $p$ and the compact set $\mathfrak C'$) such that
$$\|\phi\|_{L^p\(\D^2\)}+\|\phi\|_{L^p\(\partial\D^2\)}\le C\|\phi\|\quad\quad\quad\forall\phi\in\mathbf H_\xi,\;\xi\in\mathfrak C'.$$
\end{enumerate}
\end{lemma}

\begin{proof}
We only prove (1) because (2) and (3) follow by (1). We argue by contradiction. Assume there exist sequences $\xi_n\in\mathfrak C'$ for some $\mathfrak C'\Subset\mathfrak C$ and $\phi_n\in H^1\(\D^2\)$ such that
\begin{equation}\label{emb1}
2K(\xi_n)\int\limits_{\D^2}e^{V_{\xi_n}}\phi_n+h(\xi_n)\frac{2\varphi(\xi_n)}{\varphi(\xi_n)^2+K(\xi_n)}\int\limits_{\partial\D^2}\phi_n=0,
\end{equation}
and
\begin{equation}\label{emb2}
\|\phi_n\|_{L^2\(\D^2\)}=1\quad\hbox{and}\quad\|\nabla\phi_n\|_{L^2\(\D^2\)}\underset{n\to+\infty}\to0.
\end{equation}
Up to a subsequence, we can assume that
$$\xi_n\underset{n\to+\infty}\to\xi\in\mathfrak C,\quad\phi_n\underset{n\to+\infty}\to\phi\;\hbox{in}\;L^2\(\D^2\)\quad\hbox{and}\quad\nabla\phi_n\underset{n\to+\infty}\rightharpoonup\nabla\phi\;\hbox{in}\;L^2\(\D^2\).$$
Therefore by \eqref{emb2} we deduce that $\|\phi\|_{L^2\(\D^2\)}=1$ and $\|\nabla\phi\|_{L^2\(\D^2\)}=0$. Then $\phi\equiv\pm\frac1{\sqrt\pi}$ is constant on the disk. On the other hand, by \eqref{emb1}, taking into account that
\begin{equation}\label{key3}
2K(\xi_n)e^{V_{\xi_n}}\underset{n\to+\infty}\to2K(\xi)e^{V_\xi}\quad\hbox{uniformly in}\;\D^2
\end{equation}
and
\begin{equation}\label{key4} 
h(\xi_n)\frac{2\varphi(\xi_n)}{\varphi(\xi_n)^2+K(\xi_n)}\underset{n\to+\infty}\to h(\xi)\frac{2\varphi(\xi)}{\varphi(\xi)^2+K(\xi)}\quad\hbox{uniformly in}\;\partial\D^2,
\end{equation}
we get 
$$\pm\frac1{\sqrt\pi}\underbrace{\(2K(\xi)\int\limits_{\D^2}e^{V_\xi}+
 h(\xi)\frac{4\pi\varphi(\xi)}{\varphi(\xi)^2+K(\xi)}\)}_{\ne0\;\hbox{since}\;\xi\in\mathfrak C}=0,$$
and a contradiction arises.
\end{proof}

We also point out that a sort of Moser-Trudinger inequality holds true on $\mathbf H_\xi$:

\begin{lemma}\label{mt}
For any $\phi\in\mathbf H_\xi$ and $p>1$ one has
$$\left\|e^{|\phi|}\right\|_{L^p\(\D^2\)}+\left\|e^\frac{|\phi|}2\right\|_{L^p\(\partial\D^2\)}=O\(e^{O\(\|\phi\|^2\)}\).$$
\end{lemma}

\begin{proof}
Since $\phi\in\mathbf H_\xi$, using the definition of $V_\xi$ given in \eqref{Vbox} we have
$$8\varphi(\xi)^2K(\xi)\int\limits_{\D^2}\frac{\phi(z)}{\(\varphi(\xi)^2+K(\xi)|z|^2\)^2}dz+\frac{2h(\xi)\varphi(\xi)}{\varphi(\xi)^2+K(\xi)}\int\limits_{\partial\D^2}\phi=0;$$
therefore we can write, using Proposition \ref{estZ01},
$$\!\!\!\!\!\!\!\!\!\!\phi=\underbrace{\phi-\frac1{2\pi}\int\limits_{\partial\D^2}\phi}_{=:\phi_0}-\frac{2\varphi(\xi)^2K(\xi)\(\varphi(\xi)^2+K(\xi)\)}{\pi\(2K(\xi)+\varphi(\xi)h(\xi)\)}\underbrace{\int\limits_{\D^2}\(\phi(z)-\frac1{2\pi}\int\limits_{\partial\D^2}\phi\)\frac{dz}{\(\varphi(\xi)^2+K(\xi)|z|^2\)^2}}_{=:c}.$$
From Poincar\'e-Wirtinger inequality we get $|c|=O(\|\phi\|)$; on the other hand, since $\int\limits_{\partial\D^2}\phi_0=0$, we can estimate $\left\|e^{|\phi|}\right\|_{L^p\(\D^2\)}$ using a Moser-Trudinger type inequality from \cite{chyg}, getting
$$\int\limits_{\D^2}e^{p|\phi_0|}=O\( e^{\frac{p^2\|\phi_0\|^2}{8\pi}}\int\limits_{\D^2}e^{\frac{2\pi}{\|\phi_0\|^2}\phi_0^2}\)=e^{O\(\|\phi_0\|^2\)}\int\limits_{\D^2}e^{\frac{2\pi}{\|\phi_0\|^2}\phi_0^2}=O\(e^{O\(\|\phi_0\|^2\)}\),$$
hence
$$\left\|e^{|\phi|}\right\|_{L^p\(\D^2\)}=O\(e^{O(|c|)}\left\|e^{|\phi_0|}\right\|_{L^p\(\D^2\)}\)=O\(e^{O(\|\phi\|)}e^{O\(\|\phi_0\|^2\)}\)=O\(e^{O\(\|\phi\|^2\)}\).$$
Similarly, using the Moser-Trudinger boundary inequality from \cite{liliu} we get
$$\int\limits_{\partial\D^2}e^{p\frac{|\phi_0|}2}=O\( e^{\frac{p^2\|\phi_0\|^2}{8\pi}}\int\limits_{\partial\D^2}e^{\frac\pi{\|\phi_0\|^2}\phi_0^2}\)=e^{O\(\|\phi_0\|^2\)}\int\limits_{\partial\D^2}e^{\frac\pi{\|\phi_0\|^2}\phi_0^2}=O\(e^{O\(\|\phi_0\|^2\)}\),$$
hence $\left\|e^\frac{|\phi|}2\right\|_{L^p\(\partial\D^2\)}=O\(e^{O\(\|\phi\|^2\)}\)$.
\end{proof}

Define the functions
\begin{equation}\label{generators}
\mathcal Z_1(z):=\frac{\langle z,\xi\rangle}{\varphi(\xi)^2+K(\xi)|z|^2},\quad\quad\quad\mathcal Z_2(z):=\frac{\left\langle z,\xi^\perp\right\rangle}{\varphi(\xi)^2+K(\xi)|z|^2},
\end{equation}
that satisfy
\begin{equation}\label{eqlin}
\begin{cases}-\Delta\mathcal Z_i=2K(\xi)e^V\mathcal Z_i&\hbox{in}\;\D^2\\
\partial_\nu\mathcal Z_i=h(\xi)e^\frac V2\mathcal Z_i&\hbox{on}\;\partial\D^2,\end{cases}\quad\quad\quad i=1,2.
\end{equation}
Thus, we can state the following linear invertibility result.
\begin{theorem}\label{key}
Fix $p>1$ and $\mathfrak C'\Subset\mathfrak C$. For any $\xi\in\mathfrak C'$ and $c\in L^p\(\D^2\)$ and $d\in L^p\(\partial\D^2\)$ such that
\begin{equation}\label{cd}
\int\limits_{\D^2}c+\int\limits_{\partial\D^2}d=\int\limits_{\D^2}c\mathcal Z_i+\int\limits_{\partial\D^2}d\mathcal Z_i=0,\quad i=1,2,
\end{equation}
there exists a unique solution $\phi\in H^1\(\D^2\)$ to the problem
$$\begin{cases}-\Delta\phi=2K(\xi)e^{V_\xi}\phi+c&\hbox{in}\;\D^2\\
\partial_\nu\phi=h(\xi)e^\frac{V_\xi}2\phi+d&\hbox{on}\;\partial\D^2\\
2K(\xi)\int\limits_{\D^2}e^{V_\xi}\phi+\frac{h(\xi)2\varphi(\xi)}{\varphi(\xi)^2+K(\xi)}\int\limits_{\partial\D^2}\phi=2K(\xi)\int\limits_{\D^2}e^{V_\xi}\phi\mathcal Z_i+\frac{h(\xi)2\varphi(\xi)}{\varphi(\xi)^2+K(\xi)}\int\limits_{\partial\D^2}\phi\mathcal Z_i=0&i=1,2.\\
\end{cases}$$
Furthermore
\begin{equation}\label{key1}
\|\phi\|\le C_p\(\|c\|_{L^p\(\D^2\)}+\|d\|_{L^p\(\partial\D^2\)}\),
\end{equation}
where the constant $C_p$ only depends on $p$ and the compact set $\mathfrak C'$.
\end{theorem}

\begin{proof} 
Given $p>1$ we define the Banach space
$$\mathbf L:=\left\{(f,g)\in L^p\(\D^2\)\times L^p\(\partial\D^2\):\;\int\limits_{\D^2}f+\int\limits_{\partial\D^2}g=0\right\},$$
equipped with the norm
$$\|(f,g)\|:=\|f\|_{L^p\(\D^2\)}+\|g\|_{L^p\(\partial\D^2\)}. $$
The operator $\mathcal J_\xi:\mathbf L\to\mathbf H_\xi$ is defined by $\mathcal J_\xi(f,g)=u$ which is the unique solution in $\mathbf H_\xi$ of
$$\begin{cases}-\Delta u=f&\hbox{in}\;\D^2\\
\partial_\nu u=g&\hbox{on}\;\partial\D^2,\end{cases}$$
i.e. 
\begin{equation}\label{tstar}
\int\limits_{\D^2}\nabla u\nabla\zeta=\int\limits_{\D^2}\zeta f+\int\limits_{\partial\D^2}\zeta g\quad\quad\quad\forall\zeta\in H^1\(\D^2\).
\end{equation}
By (3) of Lemma \ref{emb} we deduce
\begin{equation}\label{nt}
\|u\|=\|\mathcal J_\xi(f,g)\|\le C_p\(\|f\|_{L^p\(\D^2\)}+\|g\|_{L^p\(\partial\D^2\)}\),
\end{equation}
where the constant $C_p$ only depends on $p,q$ and the compact set $\mathfrak C'$.

Consider now the linear problem
$$\begin{cases}-\Delta\phi=2K(\xi)e^{V_\xi}\phi+c&\hbox{in}\;\D^2\\
\partial_\nu\phi=h(\xi)e^\frac{V_\xi}2\phi+d&\hbox{on}\;\partial\D^2,\end{cases}$$
which can be rewritten as
\begin{equation}\label{linstar}
\phi=\underbrace{\mathcal J_\xi\(2K(\xi)e^{V_\xi}\phi,h(\xi)e^\frac{V_\xi}2\phi\)}_{=:\mathcal K_\xi(\phi)}+\mathcal J_\xi(c,d),\quad\hbox{i.e.}\quad\phi-\mathcal K_\xi(\phi)=\mathcal J_\xi(c,d), \end{equation}
provided $\phi\in\mathbf H_\xi$. It is important to point out that the kernel of the operator $\mathcal I-\mathcal K_\xi$ in $\mathbf H_\xi$ is a two dimensional space generated by the two functions ${\mathcal Z_\xi}_1(z)={\mathcal Z}_1(z)$, ${\mathcal Z_\xi}_2(z)={\mathcal Z}_2(z)$ defined in \eqref{generators} (we remark the dependence on $\xi$ in the subscript),
see for instance Lemma 2.3 in \cite{jlsmr}.

Set
$$\mathbf K_\xi:=\mathrm{span}\left\{{\mathcal Z_\xi}_1,{\mathcal Z_\xi}_2\right\},$$
so that
$$\mathbf K_\xi^\perp=\left\{\phi\in\mathbf H_\xi:\;\left\langle\phi,\mathcal {\mathcal Z_\xi}_i\right\rangle:=2K(\xi)\int\limits_{\D^2}e^{V_\xi}\phi{\mathcal Z_\xi}_i+\frac{h(\xi)2\varphi(\xi)}{\varphi(\xi)^2+K(\xi)}\int\limits_{\partial\D^2}\phi{\mathcal Z_\xi}_i=0,\;i=1,2.\right\}.$$
Finally, since by (2) of Lemma \ref{emb} $\mathcal K_\xi:\mathbf H_\xi\to\mathbf H_\xi$ is a compact operator, by Fredholm alternative we deduce that problem \eqref{linstar} has a unique solution $\phi\in\mathbf K_\xi^\perp$ if and only if $\mathcal J_\xi(c,d)\in\mathbf K_\xi^\perp$, i.e.
$$0=\left\langle\mathcal J_\xi(c,d),{\mathcal Z_\xi}_i\right\rangle=\int\limits_{\D^2}c{\mathcal Z_\xi}_i+\int\limits_{\partial\D^2}d{\mathcal Z_\xi}_i,\quad i=1,2.$$ 
It remains to prove that estimate \eqref{key1} is uniform with respect to the point $\xi$ in $\mathfrak C'$. We will show that there exists a positive constant $C$ such that
\begin{equation}\label{key2}
\left\|\mathcal I-\mathcal K_\xi\right\|_{\mathcal L\(\mathbf K_\xi^\perp\)}\ge C\quad\forall\xi\in\mathfrak C'.
\end{equation}
Then \eqref{key1} follows since $(\mathcal I-\mathcal K_\xi)\phi=\mathcal J_\xi(c,d)$ and by \eqref{key2} and \eqref{nt} we get
$$\|\phi\|\le\(\|\mathcal I-\mathcal K_\xi\|_{\mathcal L\(\mathbf K_\xi^\perp\)}\)^{-1}\|(\mathcal I-\mathcal K_\xi)\phi\|\le C\(\|c\|_{L^p\(\D^2\)}+\|d\|_{L^p\(\partial\D^2\)}\).$$
Let us prove \eqref{key2}. By contradiction assume there exist a sequence of points $\xi_n\in\mathfrak C'$ and sequences of functions $\phi_n,\psi_n\in\mathbf K_{\xi_n}^\perp$ such that 
$$\|\phi_n\|=1,\;\|\psi_n\|\to0\quad\hbox{and}\quad(\mathcal I-\mathcal K_{\xi_n})\phi_n=\psi_n.$$
Up to a subsequence, $\xi_n\to\xi\in\mathfrak C$, $\phi_n\to\phi$ strongly in $L^2\(\D^2\)$ and $\nabla\phi_n\rightharpoonup\nabla\phi$ weakly in $L^2\(\D^2\)$. Since \eqref{key3} and \eqref{key4} hold true, then $\phi\in\mathbf K^\perp_\xi$. Moreover by $(\mathcal I-\mathcal K_{\xi_n})\phi_n=\psi_n$ and \eqref{tstar}, for any $\zeta\in H^1\(\D^2\)$,
$$\int\limits_{\D^2}\nabla\phi_n\nabla\zeta-2K(\xi_n)\int\limits_{\D^2}e^{V_{\xi_n}}\phi_n\zeta-\frac{h(\xi_n)2\varphi(\xi_n)}{\varphi(\xi_n)^2+K(\xi_n)}\int\limits_{\partial\D^2}\phi_n\zeta=\int\limits_{\D^2}\psi_n\zeta$$
and passing to the limit
$$\int\limits_{\D^2}\nabla\phi\nabla\zeta-2K(\xi)\int\limits_{\D^2}e^{V_\xi}\phi\zeta-\frac{h(\xi)2\varphi(\xi)}{\varphi(\xi)^2+K(\xi)}\int\limits_{\partial\D^2}\phi\zeta=0,$$
i.e. $(\mathcal I-\mathcal K_\xi)\phi=0$; therefore $\phi\equiv0$, because $\phi\in\mathbf K^\perp_\xi$. On the other hand, if we take $\zeta=\phi$, we can easily deduce that
$$1-2K(\xi)\int\limits_{\D^2}e^{V_\xi(z)}\phi^2dz-\frac{h(\xi)2\varphi(\xi)}{\varphi(\xi)^2+K(\xi)}\int\limits_{\partial\D^2}\phi^2=0\quad\Rightarrow\quad\phi\not\equiv0,$$
and a contradiction arises. This concludes the proof.
\end{proof}

\medskip

\section{The nonlinear projected problem}

In order to find a solution of \eqref{eqphi}, we will first solve the associated projected problem
\begin{equation}\label{projProb}
\begin{cases}\mathcal L^\mathrm{Int}_0\phi=\mathcal E^\mathrm{Int}+\mathcal L^\mathrm{Int}\phi+\mathcal N^\mathrm{Int}(\phi)+c_0+2K(\xi)e^V(c_1\mathcal Z_1+c_2\mathcal Z_2)&\hbox{in}\;\D^2,\\
\mathcal L^\partial_0\phi=\mathcal E^\partial+\mathcal L^\partial\phi+\mathcal N^\partial(\phi)+c_0+h(\xi)e^\frac V2(c_1\mathcal Z_1+c_2\mathcal Z_2)&\hbox{on}\;\partial\D^2,
\end{cases}
\end{equation}
with $\mathcal L^\mathrm{Int}_0$, $\mathcal E^\mathrm{Int}$, $\mathcal L^\mathrm{Int}$, $\mathcal N^\mathrm{Int}$, $\mathcal L^\partial_0$, $\mathcal E^\partial$, $\mathcal L^\partial$, $\mathcal N^\partial$, defined in \eqref{eqphi}-\eqref{ebdr}, $\mathcal Z_1$, $\mathcal Z_2$ given by \eqref{generators} and $c_0$, $c_1$, $c_2\in \R$. Denoting
$$\mathcal L_0:=(\mathcal L^\mathrm{Int}_0,\mathcal L^\partial_0),\quad \mathcal{E}:=(\mathcal E^\mathrm{Int},\mathcal E^\partial),\quad \mathcal L:=(\mathcal L^\mathrm{Int},\mathcal L^\partial),\quad \mathcal{N}:=(\mathcal N^\mathrm{Int},\mathcal N^\partial),$$
we can prove the following estimates.

\begin{lemma}\label{l} Let $\phi\in\mathbf H_\xi$. Then, for any $1<p<\frac43$,
$$\|\mathcal L\phi\|_p:=\left\|\mathcal L^\mathrm{Int}\phi\right\|_{L^p\(\D^2\)}+\left\|\mathcal L^\partial\phi\right\|_{L^p\(\partial\D^2\)}=O\(\frac{|\e|}{\log\frac1{|\e|}}\|\phi\|\).$$
\end{lemma}

\begin{proof}
By definition we can write
$$\mathcal{L}^\partial=(h(f(z))-h(\xi))e^{\frac{W+\tau+V}{2}}+h(\xi)\left(e^{\frac{W+\tau}{2}}-1\right)e^{\frac{V}{2}}+\varepsilon I(f(z))e^{\frac{W+\tau+V}{2}},$$
and, proceeding similarly to the estimates of $\mathcal E^\partial_1$ and $\mathcal E^\partial_2$ in \eqref{split2}, it follows
\begin{eqnarray*}\left\|(h(f(\cdot))-h(\xi))e^{\frac{W(\cdot)+\tau+V}{2}}\right\|_{L^p(\partial\D^2)}&=&O\left(\delta^{\frac 14}+|\eta|\right),\\
\left\|h(\xi)\left(e^{\frac{W(\cdot)+\tau}{2}}-1\right)e^{\frac{V}{2}}\right\|_{L^p(\partial\D^2)}&=&O(\delta+\tau),\\
\left\|\varepsilon I(f(\cdot))e^{\frac{W(\cdot)+\tau+V}{2}}\right\|_{L^p(\partial\D^2)}&=&O\left(|\varepsilon|\left(\delta^{\frac 14}+|\eta|\right)\right).
\end{eqnarray*}
Hence, since $\mathcal L^\mathrm{Int}\phi=\mathcal E^\mathrm{Int}\phi$, by these estimates and  \eqref{eintlp} we get
\begin{eqnarray*}
\|\mathcal L\phi\|_p&=&O\(\left\|\mathcal E^\mathrm{Int}\phi\right\|_{L^p\(\D^2\)}+\left\|\mathcal L^\partial\phi\right\|_{L^p\(\partial\D^2\)}\)\\
&=&O\(\left\|\mathcal E^\mathrm{Int}\right\|_{L^{\frac32p}\(\D^2\)}\|\phi\|_{L^{3p}\(\D^2\)}+\left\|\mathcal L^\partial\right\|_{L^{\frac32p}\(\partial\D^2\)}\|\phi\|_{L^{3p}\(\partial\D^2\)}\)\\
&=&O\(\(\delta+\delta^\frac14|\e|+|\eta||\e|+|\tau|\)\|\phi\|\).
\end{eqnarray*}
\end{proof}

\begin{lemma}\label{n} Let $\phi,\phi'\in\mathbf H_\xi$. For any $p>1$,
\begin{eqnarray*}
\left\|\mathcal N(\phi)-\mathcal N(\phi')\right\|_p&:=&\left\|\mathcal N^\mathrm{Int}(\phi)-\mathcal N^\mathrm{Int}(\phi')\right\|_{L^p\(\D^2\)}+\left\|\mathcal N^\partial(\phi)-\mathcal N^\partial(\phi')\right\|_{L^p\(\partial\D^2\)}\\
&=&O\(\|\phi-\phi'\|(\|\phi\|+\|\phi'\|)e^{O\(\|\phi\|^2+\|\phi'\|^2\)}\).
\end{eqnarray*}
In particular, taking $\phi'=0$,
$$\|\mathcal N(\phi)\|_p=O\(\|\phi\|^2e^{O\(\|\phi\|^2\)}\).$$
\end{lemma}

\begin{proof}
From \eqref{wp}, Lemma \ref{emb} and Lemma \ref{mt} we get:
\begin{eqnarray*}
&&\left\|\mathcal N^\mathrm{Int}(\phi)-\mathcal N^\mathrm{Int}(\phi')\right\|_{L^p\(\D^2\)}\\
&=&\left\|2K_\e(f(z))e^{V+W+\tau}\(e^\phi-\phi-e^{\phi'}+\phi'\)\right\|_{L^p\(\D^2\)}\\
&=&O\(\left\|e^{W+\tau}\(|\phi-\phi'|(|\phi|+|\phi'|)\(1+e^{\phi+\phi'}\)\)\right\|_{L^p\(\D^2\)}\)\\
&=&O\(\left\|e^{W+\tau}\right\|_{L^{4p}\(\D^2\)}\|\phi-\phi'\|_{L^{4p}\(\D^2\)}\(\|\phi\|_{L^{4p}\(\D^2\)}+\|\phi'\|_{L^{4p}\(\D^2\)}\)\left\|1+e^{\phi+\phi'}\right\|_{L^{4p}\(\D^2\)}\)\\
&=&O\(\|\phi-\phi'\|_{L^{4p}\(\D^2\)}\(\|\phi\|_{L^{4p}\(\D^2\)}+\|\phi'\|_{L^{4p}\(\D^2\)}\)e^{O\(\|\phi+\phi'\|^2\)}\)\\
&=&O\(\|\phi-\phi'\|(\|\phi\|+\|\phi'\|)e^{O\(\|\phi+\phi'\|^2\)}\),
\end{eqnarray*}
where we used the elementary estimate
$$e^t-t-e^s+s=O\(|s-t|(|s|+|t|)\(1+e^{s+t}\)\).$$
The estimate on $\mathcal N^\partial(\phi)$ is obtained similarly.
\end{proof}\

\begin{proposition}\label{contr}
Assume $\delta,|\eta|,|\tau|,|\e|\le\e_0\ll1$. Then, there exists a unique $(\phi,c_0,c_1,c_2)\in\mathbf H_\xi\times\R^3$ such that \eqref{projProb} has a solution, which additionally satisfies
\begin{equation}\label{phio}
\|\phi\|=O\(\frac{|\e|}{\log\frac1{|\e|}}\).
\end{equation}
\end{proposition}

\begin{proof}
Let us consider the operator
$$\mathcal A:=\Pi_{\mathbf H}\circ\mathcal J_\xi\circ\Pi_{\mathbf L}\circ\mathcal L_0:\mathbf K_\xi^\perp\to\mathbf K_\xi^\perp,$$
where $\Pi_{\mathbf H}:\mathbf H_\xi\to\mathbf K_\xi^\perp$ is the standard projection in Hilbert spaces, $\mathcal J_\xi:\mathbf L\to\mathbf H_\xi$ is as in the proof of Theorem \ref{key} and $\Pi_{\mathbf L}:L^p\(\D^2\)\times L^p\(\partial\D^2\)\to\mathbf L$ is another projection defined by
$$\Pi_{\mathbf L}:(f,g)\mapsto\(f-\frac1{3\pi}\(\int\limits_{\D^2}f+\int\limits_{\partial\D^2}g\),g-\frac1{3\pi}\(\int\limits_{\D^2}f+\int\limits_{\partial\D^2}g\)\).$$
In other words, $\psi=\mathcal A\phi$ if and only if
$$\begin{cases}
-\Delta\psi=\mathcal L_0^\mathrm{Int}\phi+a_0+a_12K(\xi)e^V\mathcal Z_1+a_22K(\xi)e^V\mathcal Z_2&\hbox{in}\;\D^2\\
\partial_\nu\psi=\mathcal L_0^\partial\phi+a_0+a_1h(\xi)e^\frac V2\mathcal Z_1+a_2h(\xi)e^\frac V2\mathcal Z_2&\hbox{on}\;\partial\D^2,
\end{cases}$$
with $a_0,a_1,a_2$ suitable constants such that the right-hand side satisfies the orthogonality conditions \eqref{cd}. Thanks to Theorem \ref{key} and Sobolev embeddings, $\mathcal A$ is invertible and $\left\|\mathcal A^{-1}\phi\right\|=O(\|\phi\|)$ uniformly.

Therefore, solutions to \eqref{projProb} are fixed point of the map $\mathcal T_\xi:\mathbf K_\xi^\perp\to\mathbf K_\xi^\perp$ defined by:
$$\mathcal T_\xi:\phi\mapsto\(\mathcal A^{-1}\circ\Pi_{\mathbf H}\circ\mathcal J_\xi\circ\Pi_{\mathbf L}\)(\mathcal E+\mathcal L\phi+\mathcal N(\phi)).$$
In view of the continuity of the linear operators  $\mathcal A^{-1}$, $\Pi_{\mathbf H}$, $\mathcal J_\xi$ and $\Pi_{\mathbf L}$, and using Proposition \ref{e} and Lemmas \ref{l}, \ref{n} we have:
\begin{eqnarray}
\nonumber\|\mathcal T_\xi(\phi)\|&=&O(\|\mathcal E\|_p+\|\mathcal L\phi\|_p+\|\mathcal N(\phi)\|_p)\\
\nonumber&=&O\(\frac{|\e|}{\log\frac1{|\e|}}+\frac{|\e|}{\log\frac1{|\e|}}\|\phi\|+\|\phi\|^2e^{O\(\|\phi\|^2\)}\)\\
\label{tphi}&=&O\(\frac{|\e|}{\log\frac1{|\e|}}+\|\phi\|^2e^{O\(\|\phi\|^2\)}\),
\end{eqnarray}
and
\begin{eqnarray}
\nonumber\left\|\mathcal T_\xi(\phi)-\mathcal T_\xi(\phi')\right\|&=&O\(\left\|\mathcal L(\phi-\phi')\right\|_p+\left\|\mathcal N(\phi)-\mathcal N(\phi')\right\|_p\)\\
\label{tphi2}&=&O\(\|\phi-\phi'\|\(\frac{|\e|}{\log\frac1{|\e|}}+(\|\phi\|+\|\phi'\|)e^{O\(\|\phi\|^2+\|\phi'\|^2\)}\)\).
\end{eqnarray}
Up to taking $R$ large enough, from \eqref{tphi} we have 
$$\|\phi\|\le R\frac{|\e|}{\log\frac1{|\e|}}\quad\quad\quad\Rightarrow\quad\quad\quad\|\mathcal T_\xi(\phi)\|\le R\frac{|\e|}{\log\frac1{|\e|}};$$
if, in addition, $\delta,|\eta|,|\tau|,\e$ are small enough, then \eqref{tphi2} gives
$$\sup_{\phi\ne\phi'}\frac{\|\mathcal T_\xi(\phi)-\mathcal T_\xi(\phi')\|}{\|\phi-\phi'\|}<1,$$
for every $\phi,\phi'$ such that $\|\phi\|,\left\|\phi'\right\|\leq R\frac{|\e|}{\log\frac1{|\e|}}$. Therefore, $\mathcal T_\xi$ is a contraction on a suitable ball of $\mathbf K_\xi^\perp$, hence it has a unique fixed point that also satisfies \eqref{phio}.
\end{proof}

\medskip

\section{Estimates on the projections}

Let $\phi$ be the solution to the problem \eqref{projProb} provided by Proposition \ref{contr}. Notice that, if we prove
\begin{equation*}\label{zero}
c_0=c_1=c_2=0,
\end{equation*}
then $\phi$ actually solves \eqref{eqphi}. Hence, we need to indentify the exact expression of these constants. 

Let us multiply \eqref{projProb} by $\mathcal Z_1$ and integrate: due to the angular symmetry of $\mathcal Z_1,\mathcal Z_2$ and the radiality of $V$, we have
$$2K(\xi)\int\limits_{\D^2}e^V\mathcal Z_1\mathcal Z_2=\frac{h(\xi)2\varphi(\xi)}{\varphi(\xi)^2+K(\xi)}\int\limits_{\partial\D^2}\mathcal Z_1\mathcal Z_2=2K(\xi)\int\limits_{\D^2}e^V\mathcal Z_i=\frac{h(\xi)2\varphi(\xi)}{\varphi(\xi)^2+K(\xi)}\int\limits_{\partial\D^2}\mathcal Z_i=0\quad i=1,2,$$
and therefore
\begin{eqnarray*}
c_12K(\xi)\int\limits_{\D^2}e^V\mathcal Z_1^2&=&\int\limits_{\D^2}\mathcal L^\mathrm{Int}_0\phi\mathcal Z_1-\int\limits_{\D^2}\mathcal E^\mathrm{Int}\mathcal Z_1-\int\limits_{\D^2}\mathcal L^\mathrm{Int}\phi\mathcal Z_1-\int\limits_{\D^2}\mathcal N^\mathrm{Int}(\phi)\mathcal Z_1,\\
c_1\frac{h(\xi)2\varphi(\xi)}{\varphi(\xi)^2+K(\xi)}\int\limits_{\partial\D^2}\mathcal Z_1^2&=&\int\limits_{\partial\D^2}\mathcal L^\partial_0\phi\mathcal Z_1-\int\limits_{\partial\D^2}\mathcal E^\partial\mathcal Z_1-\int\limits_{\partial\D^2}\mathcal L^\partial\phi\mathcal Z_1-\int\limits_{\partial\D^2}\mathcal N^\partial(\phi)\mathcal Z_1.
\end{eqnarray*}
Integrating by parts and using \eqref{eqlin},
$$\int\limits_{\D^2}\(\mathcal L^\mathrm{Int}_0\phi\)\mathcal Z_1+\int\limits_{\partial\D^2}\(\mathcal L^\partial_0\phi\)\mathcal Z_1=\int\limits_{\D^2}\(\mathcal L^\mathrm{Int}_0\mathcal Z_1\)\phi+\int\limits_{\partial\D^2}\(\mathcal L^\partial_0\mathcal Z_1\)\phi=0,$$
and thus
\begin{eqnarray*}
&&c_1\(2K(\xi)\int\limits_{\D^2}e^V\mathcal Z_1^2+\frac{h(\xi)2\varphi(\xi)}{\varphi(\xi)^2+K(\xi)}\int\limits_{\partial\D^2}\mathcal Z_1^2\)\\
&=&-\int\limits_{\D^2}\mathcal E^\mathrm{Int}\mathcal Z_1-\int\limits_{\D^2}\mathcal L^\mathrm{Int}\phi\mathcal Z_1-\int\limits_{\D^2}\mathcal N^\mathrm{Int}(\phi)\mathcal Z_1-\int\limits_{\partial\D^2}\mathcal E^\partial\mathcal Z_1-\int\limits_{\partial\D^2}\mathcal L^\partial\phi\mathcal Z_1-\int\limits_{\partial\D^2}\mathcal N^\partial(\phi)\mathcal Z_1.
\end{eqnarray*}
A similar formula holds true for $\mathcal Z_2$. Moreover, since $\phi\in\mathbf H_\xi$, then
$$\int\limits_{\D^2}\mathcal L^\mathrm{Int}_0\phi+\int\limits_{\partial\D^2}\mathcal L^\partial_0\phi=0,$$
and integrating \eqref{projProb} gives
$$3\pi c_0=-\int\limits_{\D^2}\mathcal E^\mathrm{Int}-\int\limits_{\D^2}\mathcal L^\mathrm{Int}\phi-\int\limits_{\D^2}\mathcal N^\mathrm{Int}(\phi)-\int\limits_{\partial\D^2}\mathcal E^\partial-\int\limits_{\partial\D^2}\mathcal L^\partial\phi-\int\limits_{\partial\D^2}\mathcal N^\partial(\phi).$$
Recalling \eqref{orders}, let us estimate all the integral terms involved.

\begin{proposition}\label{EintZ1}
\begin{eqnarray*}
\int\limits_{\D^2}\mathcal E^\mathrm{Int}\mathcal Z_1&=&\frac{4\pi}{\(\varphi(\xi)^2+K(\xi)\)^2}\delta\(\partial_1K(1)+\eta\partial_{12}K(1)+\e\partial_1G(1)-\delta\log\frac1\delta\frac{2\varphi(1)^2}{\varphi(1)^2+K(1)}\Delta K(1)\)\\
&+&\int\limits_{\D^2}2K(\xi)We^V\mathcal Z_1+O\(\frac{|\e|^2}{\log^2\frac1{|\e|}}\).
\end{eqnarray*}
\end{proposition}

\begin{proof}
We split $\mathcal E^\mathrm{Int}=\mathcal E^\mathrm{Int}_1+\mathcal E^\mathrm{Int}_2+\mathcal E^\mathrm{Int}_3$ as in \eqref{split1} and we start by estimating the first term. We can write:
\begin{eqnarray}
\nonumber\mathcal E^\mathrm{Int}_1&=&2K(\xi)(W+\tau)e^V+2(K(f(z))-K(\xi))(W+\tau)e^V\\
&+&2(K(f(z))+\e G(f(z)))\(e^{W+\tau}-1-(W+\tau)\)e^V\nonumber\\
\label{eint1}&=&2K(\xi)We^V+2K(\xi)\tau e^V+O\(|f(z)-\xi|(|W|+|\tau|)+(|W|+|\tau|)^2\(1+e^{W+\tau}\)\).
\end{eqnarray}
Since $V$ is radial, we have $\int\limits_{\D^2}e^V\mathcal Z_1=0$; therefore, in view of Lemma \ref{w} and Proposition \ref{f},
\begin{eqnarray}
\nonumber\int\limits_{\D^2}\mathcal E^\mathrm{Int}_1\mathcal Z_1&=&2K(\xi)\int\limits_{\D^2}We^V\mathcal Z_1+2K(\xi)\tau\int\limits_{\D^2}e^V\mathcal Z_1\\
&+&\int\limits_{\D^2}O\(|f(z)-\xi|(|W|+|\tau|)+(|W|+|\tau|)^2\(1+e^{W+\tau}\)\)\nonumber\\
\label{eint1z1}&=&\int\limits_{\D^2}2K(\xi)We^V\mathcal Z_1+O\(\delta^2+\delta|\tau|+|\tau|^2\).
\end{eqnarray}
To estimate the term with $\mathcal E^\mathrm{Int}_2$, we make the following Taylor expansion:
\begin{eqnarray}
\nonumber\mathcal E^\mathrm{Int}_2&=&2\langle\nabla K(\xi),f(z)-\xi\rangle e^V+\left\langle D^2K(\xi)(f(z)-\xi),f(z)-\xi\right\rangle e^V+2\e\langle\nabla G(\xi),f(z)-\xi\rangle e^V\\
\!\!\!\!\!\!\!\!\!\!\!\!\!\!\!\!\!\!\!\!\nonumber&+&O\(|f(z)-\xi|^3+|\e||f(z)-\xi|^2\)\\
\!\!\!\!\!\!\!\!\!\!\!\!\!\!\!\!\!\!\!\!\nonumber&=&2\delta\left\langle\nabla K(\xi),\frac{z-\xi}{1+(1-\delta)\overline\xi z}\right\rangle e^V+\delta^2\left\langle D^2K(\xi)\frac{z-\xi}{1+(1-\delta)\overline\xi z},\frac{z-\xi}{1+(1-\delta)\overline\xi z}\right\rangle e^V\\
\!\!\!\!\!\!\!\!\!\!\!\!\!\!\!\!\!\!\!\!\label{eint2}&+&2\delta\e\left\langle\nabla G(\xi),\frac{z-\xi}{1+(1-\delta)\overline\xi z}\right\rangle e^V+O\(\frac{\delta^3}{(\delta+|z+\xi|)^3}+\frac{|\e|\delta^2}{(\delta+|z+\xi|)^2}\),
\end{eqnarray}
where we used the definition of $f(z)$ and \eqref{fxi}.

The terms involving the first derivatives can be handled by making a rotation $w=\overline\xi z$ and then using Proposition \ref{estKgrad}:
\begin{eqnarray*}
&&\int\limits_{\D^2}\left\langle\nabla K(\xi),\frac{z-\xi}{1+(1-\delta)\overline\xi z}\right\rangle e^V\mathcal Z_1\\
&=&\int\limits_{\D^2}\left\langle\nabla K(\xi),\xi\frac{w-1}{1+(1-\delta)w}\right\rangle\frac{4\varphi(\xi)^2}{\(\varphi(\xi)^2+K(\xi)|w|^2\)^3}w_1dw\\
&=&\int\limits_{\D^2}\left\langle\nabla K(\xi),\xi\frac{(1-\delta)^2|w|^2-1+\delta w_1+\imath(2-\delta)w_2}{|1+(1-\delta)w|^2}\right\rangle\frac{4\varphi(\xi)^2}{\(\varphi(\xi)^2+K(\xi)|w|^2\)^3}w_1dw\\
&=&\langle\nabla K(\xi),\xi\rangle\int\limits_{\D^2}\frac{(1-\delta)^2|w|^2-1+\delta w_1}{|1+(1-\delta)w|^2}\frac{4\varphi(\xi)^2}{\(\varphi(\xi)^2+K(\xi)|w|^2\)^3}w_1dw\\
&+&\left\langle\nabla K(\xi),\xi^\perp\right\rangle\int\limits_{\D^2}\frac{(2-\delta)w_2}{|1+(1-\delta)w|^2}\frac{4\varphi(\xi)^2}{\(\varphi(\xi)^2+K(\xi)|w|^2\)^3}w_1dw\\
&=&\(\frac{2\pi}{\(\varphi(\xi)^2+K(\xi)\)^2}+O(\delta)\)\langle\nabla K(\xi),\xi\rangle\\
&=&\(\frac{2\pi}{\(\varphi(\xi)^2+K(\xi)\)^2}+O(\delta)\)\left\langle\nabla K(1)+\eta(\partial_{12}K(1)+\imath\partial_{22}K(1))+O\(|\eta|^2\),1+O(|\eta|)\right\rangle\\
&=&\frac{2\pi}{\(\varphi(\xi)^2+K(\xi)\)^2}(\partial_1K(1)+\eta\partial_{12}K(1))+O\(\delta+|\eta|^2\).
\end{eqnarray*}
Likewise,
$$\int\limits_{\D^2}\langle\nabla G(\xi),f(z)-\xi\rangle e^V\mathcal Z_1=\frac{2\pi}{\(\varphi(\xi)^2+K(\xi)\)^2}\partial_1G(1)+O(\delta+|\eta|).$$
To estimate the second order integral we first observe that
$$\frac{z-\xi}{1+(1-\delta)\overline\xi z}=\frac{-2\xi}{1+(1-\delta)\overline\xi z}+O\(\frac{|z+\xi|}{\delta+|z+\xi|}\),$$
therefore
\begin{eqnarray*}
&&\int\limits_{\D^2}\left\langle D^2K(\xi)\frac{z-\xi}{1+(1-\delta)\overline\xi z},\frac{z-\xi}{1+(1-\delta)\overline\xi z}\right\rangle e^V\mathcal Z_1\\
&=&\int\limits_{\D^2}\(\left\langle D^2K(\xi)\frac{-2\xi}{1+(1-\delta)\overline\xi z},\frac{-2\xi}{1+(1-\delta)\overline\xi z}\right\rangle+O\(\frac{|z+\xi|}{(\delta+|z+\xi|)^2}\)\)\(\frac{-4\varphi(\xi)^2}{\(\varphi(\xi)^2+K(\xi)\)^3}+O(|z+\xi|)\)dz\\
&=&-\frac{16\varphi(\xi)^2}{\(\varphi(\xi)^2+K(\xi)\)^3}\int\limits_{\D^2}\(\left\langle D^2K(\xi)\frac\xi{1+(1-\delta)\overline\xi z},\frac\xi{1+(1-\delta)\overline\xi z}\right\rangle dz+O\(\frac{|z+\xi|}{(\delta+|z+\xi|)^2}\)\)dz\\
&=&-\frac{16\varphi(\xi)^2}{\(\varphi(\xi)^2+K(\xi)\)^3}\frac1{(1-\delta)^2}\int\limits_{\Omega_\delta}\left\langle D^2K(\xi)\frac\xi{1+w},\frac\xi{1+w}\right\rangle dw+O(1),
\end{eqnarray*}
where we used the change of variable $w=\(\frac{1-\delta}\delta\)\(\overline\xi z+1\)$, which transforms the disk $\D^2$ into
$$\Omega_\delta:=\left\{\left|w-\frac{1-\delta}\delta\right|\le\frac{1-\delta}\delta\right\},$$
and 
$$\int\limits_{\D^2}\left\langle D^2K(\xi)\frac\xi{1+(1-\delta)\overline\xi z},\frac\xi{1+(1-\delta)\overline\xi z}\right\rangle dz=\frac1{(1-\delta)^2}\int\limits_{\Omega_\delta}\left\langle D^2K(\xi)\frac\xi{1+w},\frac\xi{1+w}\right\rangle dw.$$
Using the formulas
$$\frac\xi{1+w}=\frac{e^{\imath\eta}(1+\overline w)}{|1+w|^2}=\frac1{\((1+w_1)^2+w_2^2\)^2}\((1+w_1)\cos\eta+w_2\sin\eta+\imath((1+w_1)\sin\eta-w_2\cos\eta)\),$$
$$\left\langle D^2K(\xi)(z_1+\imath z_2),z_1+\imath z_2\right\rangle=\frac{\Delta K(\xi)}2|z|^2+\frac{\partial_{11}K(\xi)-\partial_{22}K(\xi)}2\(z_1^2-z_2^2\)+2\partial_{12}K(\xi)z_1z_2,$$
and Proposition \ref{estKlap} we get:
\begin{eqnarray*}
&&\int\limits_{\Omega_\delta}\left\langle D^2K(\xi)\frac\xi{1+w},\frac\xi{1+w}\right\rangle\\
&=&\frac{\Delta K(\xi)}2\int\limits_{\Omega_\delta}\frac{dw}{(1+w_1)^2+w_2^2}\\
&+&\(\frac{\cos^2\eta-\sin^2\eta}2(\partial_{11}K(\xi)-\partial_{22}K(\xi))+\cos\eta\sin\eta\partial_{12}K(\xi)\)\int\limits_{\Omega_\delta}\frac{(1+w_1)^2-w_2^2}{\((1+w_1)^2+w_2^2\)^2}dw\\
&+&\(\cos\eta\sin\eta(\partial_{11}K(\xi)-\partial_{22}K(\xi))+\(\cos^2\eta-\sin^2\eta\)\partial_{12}K(\xi)\)\int\limits_{\Omega_\delta}\frac{(1+w_1)w_2}{\((1+w_1)^2+w_2^2\)^2}dw\\
&=&\frac\pi2\Delta K(\xi)\log\frac1\delta+O(1).
\end{eqnarray*}
Therefore,
\begin{eqnarray*}
&&\int\limits_{\D^2}\left\langle D^2K(\xi)\frac{z-\xi}{1+(1-\delta)\overline\xi z},\frac{z-\xi}{1+(1-\delta)\overline\xi z}\right\rangle e^V\mathcal Z_1\\
&=&-\frac{16\pi\varphi(\xi)^2}{\(\varphi(\xi)^2+K(\xi)\)^3}(1+O(\delta))\(\frac\pi2\Delta K(\xi)\log\frac1\delta+O(1)\)+O(1)\\
&=&-\frac{8\pi\varphi(\xi)^2}{\(\varphi(\xi)^2+K(\xi)\)^3}\Delta K(1)\log\frac1\delta+O(1);
\end{eqnarray*}
since we have
\begin{eqnarray*}
\int\limits_{\D^2}\frac{\delta^3}{(\delta+|z+\xi|)^3}\mathcal Z_1&=&O\(\int\limits_{\D^2}\frac{\delta^3}{(\delta+|z+\xi|)^3}\)=O\(\delta^2\)\\
\int\limits_{\D^2}\frac{|\e|\delta^2}{(\delta+|z+\xi|)^2}\mathcal Z_1&=&O\(\int\limits_{\D^2}\frac{|\e|\delta^2}{(\delta+|z+\xi|)^2}\)=O\(\delta^2\log\frac1\delta|\e|\),
\end{eqnarray*}
then, \eqref{eint2} gives
\begin{eqnarray}
\nonumber\int\limits_{\D^2}\mathcal E^\mathrm{Int}_2\mathcal Z_1&=&\frac{4\pi}{\(\varphi(\xi)^2+K(\xi)\)^2}\delta\(\partial_1K(1)+\eta\partial_{12}K(1)+\e\partial_1G(1)-\delta\log\frac1\delta\frac{2\varphi(1)^2}{\varphi(1)^2+K(1)}\Delta K(1)\)\\
\!\!\!\!\!\!\!\!\!\!\!\!\!\!\!\!\!\!\!\!\label{eint2z1}&+&O\(\delta\(\delta+|\eta|^2+\delta\log\frac1\delta|\eta|+\delta\log\frac1\delta|\e|+|\eta||\e|\)\).
\end{eqnarray}
Finally, since $e^V$ is radially symmetric and $\mathcal Z_1$ is odd with respect to $\xi$, then
$$\int\limits_{\D^2}\mathcal E^\mathrm{Int}_3\mathcal Z_1=0,$$
and putting it together with \eqref{eint1z1} and \eqref{eint2z1} the proof is concluded.
\end{proof}\

\begin{proposition}\label{EbdrZ1}
\begin{eqnarray*}
\int\limits_{\partial\D^2}\mathcal E^\partial\mathcal Z_1&=&\frac{8\pi\varphi(\xi)}{\(\varphi(\xi)^2+K(\xi)\)^2}\delta\((-\Delta)^\frac12h(1)+\eta(-\Delta)^\frac12h'(1)+\e(-\Delta)^\frac12I(1)\right.\\
&-&\left.\delta\log\frac1\delta\frac{4\varphi(1)}{\varphi(1)^2+K(1)}\((-\Delta)^\frac12h(1)^2+h'(1)^2\)\)-2K(\xi)\int\limits_{\D^2}We^V\mathcal Z_1\\
&+&O\(\frac{|\e|^2}{\log^2\frac1{|\e|}}\).
\end{eqnarray*}
\end{proposition}

\begin{proof}
We split $\mathcal E^\partial=\mathcal E^\partial_1+\mathcal E^\partial_2+\mathcal E^\partial_3$ as in \eqref{split2}.
To estimate the first term we see that
\begin{eqnarray}
\nonumber\mathcal E^\partial_1&=&h(\xi)(W+\tau)e^\frac V2+(h(f(z))-h(\xi))(W+\tau)e^\frac V2+2h(f(z))\(e^\frac{W+\tau}2-1-\frac{W+\tau}2\)e^\frac V2\\
\!\!\!\!\!\!\!\!\!\!\nonumber&=&h(\xi)We^\frac V2+\tau h(\xi)e^\frac V2+(h(f(z))-h(\xi))We^\frac V2+\tau(h(f(z))-h(\xi))e^\frac V2\\
\!\!\!\!\!\!\!\!\!\!\label{ebdr1}&+&O\((|W|+|\tau|)^2\(1+e^{W+\tau}\)\);
\end{eqnarray}
since $W,\mathcal Z_1$ are solutions to \eqref{eqw} and \eqref{eqlin} respectively, an integration by parts gives:
\begin{eqnarray*}
&&\frac{h(\xi)2\varphi(\xi)}{\varphi(\xi)^2+K(\xi)}\int\limits_{\partial\D^2}W\mathcal Z_1\\
&=&\int\limits_{\partial\D^2}W\partial_\nu\mathcal Z_1\\
&=&\int\limits_{\partial\D^2}\mathcal Z_1\partial_\nu W+\int\limits_{\D^2}(W\Delta\mathcal Z_1-\mathcal Z_1\Delta W)\\
&=&\frac{2\varphi(\xi)}{\varphi(\xi)^2+K(\xi)}\int\limits_{\partial\D^2}2(h(f(z))-h(\xi))\mathcal Z_1-\frac1\pi \frac{2\varphi(\xi)}{\varphi(\xi)^2+K(\xi)}\int\limits_{\partial\D^2}(h(f(w))-h(\xi))dw\int\limits_{\partial\D^2}\mathcal Z_1\\
&&-2K(\xi)\int\limits_{\D^2}We^V\mathcal Z_1\\
&=&\frac{2\varphi(\xi)}{\varphi(\xi)^2+K(\xi)}\int\limits_{\partial\D^2}2(h(f(z))-h(\xi))\mathcal Z_1-2K(\xi)\int\limits_{\D^2}We^V\mathcal Z_1.
\end{eqnarray*}
The first integral can be estimated by means of Proposition \ref{intB}:
\begin{eqnarray*}
&&\frac{2\varphi(\xi)}{\varphi(\xi)^2+K(\xi)}\int\limits_{\partial\D^2}(h(f(z))-h(\xi))\mathcal Z_1\\
&=&\frac{2\varphi(\xi)}{\(\varphi(\xi)^2+K(\xi)\)^2}\int\limits_{\partial\D^2}(h(f(z))-h(\xi))\langle z,\xi\rangle dz\\
&=&\frac{4\pi\varphi(\xi)}{\(\varphi(\xi)^2+K(\xi)\)^2}\delta(-\Delta)^\frac12h(\xi)+O\(\delta^2\)\\
&=&\frac{4\pi\varphi(\xi)}{\(\varphi(\xi)^2+K(\xi)\)^2}\delta\((-\Delta)^\frac12h(1)+\eta(-\Delta)^\frac12h'(1)\)+O\(\delta\(\delta+|\eta|^2\)\).
\end{eqnarray*}
On the other hand, for the estimates involving $W$ we will use  \eqref{wi} and then Proposition \ref{logarctan}:
\begin{eqnarray*}
&&\frac{2\varphi(\xi)}{\varphi(\xi)^2+K(\xi)}\int\limits_{\partial\D^2}(h(f(z))-h(\xi))W\mathcal Z_1\\
&=&\frac{2\varphi(\xi)}{\(\varphi(\xi)^2+K(\xi)\)^2}\int\limits_{\partial\D^2}(h(f(z))-h(\xi))W(z)\langle z,\xi\rangle dz\\
&=&\frac{16\varphi(\xi)^2}{\(\varphi(\xi)^2+K(\xi)\)^3}\delta\((-\Delta)^\frac12h(\xi)\int\limits_{\partial\D^2}(h(f(z))-h(\xi))\log|z+\xi|\langle z,\xi\rangle dz\right.\\
&+&\left.h'(\xi)\int\limits_{\partial\D^2}(h(f(z))-h(\xi))\arctan\frac{\left\langle z,\xi^\perp\right\rangle}{1+\langle z,\xi\rangle}\langle z,\xi\rangle dz\)\\
&+&\frac{2\varphi(\xi)}{\(\varphi(\xi)^2+K(\xi)\)^2}\int\limits_{\partial\D^2}(h(f(z))-h(\xi))O\(\frac{\delta^2}{\delta+|z+\xi|}\(1+\left|\log\frac{|z+\xi|}\delta\right|\)\)\langle z,\xi\rangle dz\\
&=&\frac{-32\pi\varphi(\xi)^2}{\(\varphi(\xi)^2+K(\xi)\)^3}\delta\(\((-\Delta)^\frac12h(\xi)^2+h'(\xi)^2\)\delta\log\frac1\delta+O(\delta)\)\\
&+&O\(\int\limits_{\partial\D^2}\frac{\delta^3}{\(\delta+|z+\xi|\)^2}\(1+\left|\log\frac{|z+\xi|}\delta\right|\)dz\)\\
&=&\frac{-32\pi\varphi(\xi)^2}{\(\varphi(\xi)^2+K(\xi)\)^3}\delta^2\log\frac1\delta\((-\Delta)^\frac12h(\xi)^2+h'(\xi)^2\)+O\(\delta^2\)+O\(\int\limits_0^{+\infty}\frac{\delta^2}{(1+t)^2}(1+|\log t|)dt\)\\
&=&\frac{-32\pi\varphi(\xi)^2}{\(\varphi(1)^2+K(1)\)^3}\delta^2\log\frac1\delta\((-\Delta)^\frac12h(1)^2+h'(1)^2\)+O\(\delta^2+\delta^2\log\frac1\delta|\eta|\).
\end{eqnarray*}
Moreover, since $e^\frac V2$ is constant on $\partial\D^2$, then $\int\limits_{\partial\D^2}e^\frac V2\mathcal Z_1=0$, and therefore the decomposition \eqref{ebdr1} gives:
\begin{eqnarray}
\nonumber\int\limits_{\partial\D^2}\mathcal E^\partial_1\mathcal Z_1&=&(2+\tau)\frac{2\varphi(\xi)}{\varphi(\xi)^2+K(\xi)}\int\limits_{\partial\D^2}(h(f(z))-h(\xi))\mathcal Z_1+\frac{2\varphi(\xi)}{\varphi(\xi)^2+K(\xi)}\int\limits_{\partial\D^2}(h(f(z))-h(\xi))W\mathcal Z_1\\
\nonumber &-&2K(\xi)\int\limits_{\D^2}We^V\mathcal Z_1+\int\limits_{\partial\D^2}O\((|W|+|\tau|)^2\(1+e^{W+\tau}\)\)\mathcal Z_1\\
\nonumber&=&(2+O(|\tau|))\frac{4\pi\varphi(\xi)}{\(\varphi(\xi)^2+K(\xi)\)^2}\delta\((-\Delta)^\frac12h(1)+\eta(-\Delta)^\frac12h'(1)\)+O\(\delta\(\delta+|\eta|^2\)\)\\
\nonumber&-&\frac{-32\pi\varphi(\xi)^2}{\(\varphi(\xi)^2+K(\xi)\)^3}\delta^2\log\frac1\delta\((-\Delta)^\frac12h(1)^2+h'(1)^2\)O\(\delta^2+\delta^2\log\frac1\delta|\eta|\)\\
\nonumber&-&2K(\xi)\int\limits_{\D^2}We^V\mathcal Z_1+\int\limits_{\partial\D^2}\(\delta\(1+\log\frac1{|z+\xi|}\)+|\tau|\)^2\(1+\frac1{|z+\xi|^{O(\delta)}}\)\\
\nonumber&=&\frac{8\pi\varphi(\xi)}{\(\varphi(\xi)^2+K(\xi)\)^2}\delta\((-\Delta)^\frac12h(1)+\eta(-\Delta)^\frac12h'(1)\right.\\
\nonumber&-&\left.\delta\log\frac1\delta\frac{4\varphi(1)^2}{\varphi(1)^2+K(1)}\((-\Delta)^\frac12h(\xi)^2+h'(\xi)^2\)\)\\
\label{ebdr1z1}&-&2K(\xi)\int\limits_{\D^2}We^V\mathcal Z_1+O\(\delta^2+\delta|\tau|+\delta|\eta|^2+\delta^2\log\frac1\delta|\eta|+|\tau|^2\).
\end{eqnarray}
To deal with $\mathcal E^\partial_2$, we use Proposition \ref{intB} similarly as before:
\begin{eqnarray}
\!\!\!\!\!\!\!\!\!\!\nonumber\int\limits_{\partial\D^2}\mathcal E^\partial_2\mathcal Z_1&=&2\e \frac{2\varphi(\xi)}{\varphi(\xi)^2+K(\xi)}\int\limits_{\partial\D^2}(I(f(z))-I(\xi))\mathcal Z_1+2\e \frac{2\varphi(\xi)}{\varphi(\xi)^2+K(\xi)}I(\xi)\int\limits_{\partial\D^2}\mathcal Z_1\\
\!\!\!\!\!\!\!\!\!\!\nonumber&=&2\e\(\frac{4\pi\varphi(\xi)}{\(\varphi(\xi)^2+K(\xi)\)^2}\delta(-\Delta)^\frac12I(\xi)+O\(\delta^2\)\)+\frac{4\e I(\xi)\varphi(\xi)}{\(\varphi(\xi)^2+K(\xi)\)^2}\int\limits_{\partial\D^2}\langle z,\xi\rangle dz\\
\!\!\!\!\!\!\!\!\!\!\label{ebdr2z1}&=&\frac{8\pi\varphi(\xi)}{\(\varphi(\xi)^2+K(\xi)\)^2}\delta\e(-\Delta)^\frac12I(1)+O\(\delta|\e|(\delta+|\eta|)\).
\end{eqnarray}
Finally, since $\mathcal E^\partial_3$ is constant, then $\int\limits_{\partial\D^2}\mathcal E^\partial_3\mathcal Z_1=0$; therefore, the proof follows by putting together \eqref{ebdr1z1} and \eqref{ebdr2z1}.
\end{proof}\

\begin{proposition}\label{EintZ2}
\begin{eqnarray*}
\int\limits_{\D^2}\mathcal E^\mathrm{Int}\mathcal Z_2&=&\frac{4\pi}{\(\varphi(\xi)^2+K(\xi)\)^2}\delta(\partial_2K(1)+\eta\partial_{22}K(1)+\e\partial_2G(1))+2K(\xi)\int\limits_{\D^2}We^V\mathcal Z_2\\
&+&O\(\frac{|\e|^2}{\log^2\frac1{|\e|}}\).
\end{eqnarray*}
\end{proposition}

\begin{proof}
The proof is similar to  \eqref{EintZ1}, so we will be sketchy.

We write, as before, $\mathcal E^\mathrm{Int}=\mathcal E^\mathrm{Int}_1+\mathcal E^\mathrm{Int}_2+\mathcal E^\mathrm{Int}_3$. For the first term we have:
$$\int\limits_{\D^2}\mathcal E^\mathrm{Int}_1\mathcal Z_2=2K(\xi)\int\limits_{\D^2}We^V\mathcal Z_2+O\(\delta^2+|\tau|^2\).$$
The second term is split as in \eqref{eint2}; the terms with the gradient are estimated using Proposition \ref{estKgrad}:
\begin{eqnarray*}
&&\int\limits_{\D^2}\left\langle\nabla K(\xi),\frac{z-\xi}{1+(1-\delta)\overline\xi z}\right\rangle e^V\mathcal Z_2\\&=&\langle\nabla K(\xi),\xi\rangle\int\limits_{\D^2}\frac{(1-\delta)^2|w|^2-1+\delta w_1}{|1+(1-\delta)w|^2}\frac{4\varphi(\xi)^2}{\(\varphi(\xi)^2+K(\xi)|w|^2\)^3}w_2dw\\
&+&\left\langle\nabla K(\xi),\xi^\perp\right\rangle\int\limits_{\D^2}\frac{(2-\delta)w_2}{|1+(1-\delta)w|^2}\frac{4\varphi(\xi)^2}{\(\varphi(\xi)^2+K(\xi)|w|^2\)^3}w_2dw\\
&=&\(\frac{2\pi}{\(\varphi(\xi)^2+K(\xi)\)^2}+O(\delta)\)\left\langle\nabla K(\xi),\xi^\perp\right\rangle\\
&=&\frac{2\pi}{\(\varphi(\xi)^2+K(\xi)\)^2}(\partial_2K(1)+\eta\partial_{22}K(1))+O\(\delta+|\eta|^2\),
\end{eqnarray*}
and similarly
$$\int\limits_{\D^2}\langle\nabla G(\xi),f(z)-\xi\rangle e^V\mathcal Z_2=\frac{2\pi}{\(\varphi(\xi)^2+K(\xi)\)^2}\partial_2G(1)+O(\delta+|\eta|).$$
The term with the second derivatives this time is negligible because $\left\langle z,\xi^\perp\right\rangle=O(|z+\xi|)$, hence
\begin{eqnarray*}
\int\limits_{\D^2}\left\langle D^2K(\xi)\frac{z-\xi}{1+(1-\delta)\overline\xi z},\frac{z-\xi}{1+(1-\delta)\overline\xi z}\right\rangle e^V\mathcal Z_2
&=&\int\limits_{\D^2}O\(\left|\frac{z-\xi}{1+(1-\delta)\overline\xi z}\right|^2\left|\left\langle z,\xi^\perp\right\rangle\right|\)\\
&=&\int\limits_{\D^2}\frac{|z+\xi|}{(|z+\xi|+\delta)^2}\\
&=&O(1);
\end{eqnarray*}
therefore,
$$\!\!\!\!\!\!\!\!\!\!\int\limits_{\D^2}\mathcal E^\mathrm{Int}_2\mathcal Z_2=\frac{4\pi}{\(\varphi(\xi)^2+K(\xi)\)^2}\delta\(\partial_2K(1)+\eta\partial_{22}K(1)+\e\partial_2G(1)\)+O\(\delta\(\delta+|\eta|^2+|\eta||\e|\)\).$$
Finally, $\int\limits_{\D^2}\mathcal E^\mathrm{Int}_3\mathcal Z_2=0$ again by symmetry reasons, which concludes the proof.
\end{proof}\

\begin{proposition}\label{EbdrZ2}
\begin{eqnarray*}
\int\limits_{\partial\D^2}\mathcal E^\partial\mathcal Z_2&=&\frac{8\pi\varphi(\xi)}{\(\varphi(\xi)^2+K(\xi)\)^2}\delta(h'(1)+\eta h''(1)+\e I'(1))-2K(\xi)\int\limits_{\D^2}We^V\mathcal Z_2\\
&+&O\(\frac{|\e|^2}{\log^2\frac1{|\e|}}\).
\end{eqnarray*}
\end{proposition}

\begin{proof}
The proof is similar to Proposition \ref{EbdrZ1}.

We split $\mathcal E^\partial$ as in \eqref{ebdr}. The term with $\mathcal E^\partial_1$ is again divided as in \eqref{ebdr1} and an integration by parts gives
$$h(\xi)\frac{2\varphi(\xi)}{\varphi(\xi)^2+K(\xi)}\int\limits_{\partial\D^2}W\mathcal Z_2=2\frac{2\varphi(\xi)}{\varphi(\xi)^2+K(\xi)}\int\limits_{\partial\D^2}(h(f(z))-h(\xi))\mathcal Z_2-2K(\xi)\int\limits_{\D^2}We^V\mathcal Z_2;$$
in turn, the first term can be estimated via Proposition \ref{intB}:
\begin{eqnarray*}
\frac{2\varphi(\xi)}{\varphi(\xi)^2+K(\xi)}\int\limits_{\partial\D^2}(h(f(z))-h(\xi))\mathcal Z_2
&=&\frac{4\pi\varphi(\xi)}{\(\varphi(\xi)^2+K(\xi)\)^2}\delta h'(\xi)+O\(\delta^2\)\\
&=&\frac{4\pi\varphi(\xi)}{\(\varphi(\xi)^2+K(\xi)\)^2}\delta(h'(1)+\eta h''(1))+O\(\delta\(\delta+|\eta|^2\)\).
\end{eqnarray*}
The term containing $W$ is again estimated using   \eqref{w} and Proposition \ref{logarctan}:
\begin{eqnarray*}
\frac{2\varphi(\xi)}{\varphi(\xi)^2+K(\xi)}\int\limits_{\partial\D^2}(h(f(z))-h(\xi))W\mathcal Z_2&=&O\(\delta^2\)+O\(\int\limits_{\partial\D^2}\frac{\delta^3}{\(\delta+|z+\xi|\)^2}\(1+\left|\log\frac{|z+\xi|}\delta\right|\)dz\)\\
&=&O\(\delta^2\);
\end{eqnarray*}
therefore,
\begin{eqnarray*}
\int\limits_{\partial\D^2}\mathcal E^\partial_1\mathcal Z_2&=&\frac{8\pi\varphi(\xi)}{\(\varphi(\xi)^2+K(\xi)\)^2}\delta(h'(1)+\eta h''(1))-2K(\xi)\int\limits_{\D^2}We^V\mathcal Z_2+O\(\delta^2+\delta|\tau|+\delta|\eta|^2+|\tau|^2\).
\end{eqnarray*}
Using similarly Proposition \ref{intB} we get:
$$\int\limits_{\partial\D^2}\mathcal E^\partial_2\mathcal Z_2=\frac{8\pi\varphi(\xi)}{\(\varphi(\xi)^2+K(\xi)\)^2}\delta\e I'(1)+O\(\delta|\e|(\delta+|\eta|)\).$$
Finally, $\int\limits_{\partial\D^2}\mathcal E^\partial_3\mathcal Z_2=0$ and the proof is concluded.
\end{proof}\

\begin{proposition}\label{eintz0}
$$\int\limits_{\D^2}\mathcal E^\mathrm{Int}=\frac{8\pi}{\varphi(\xi)^2+K(\xi)}\(-\delta\partial_1K(1)+K(1)\tau\)+O\(\frac{|\e|^2}{\log\frac1{|\e|}}\).$$
\end{proposition}

\begin{proof}
We split $\mathcal E^\mathrm{Int}$ as in \eqref{split1}, and $\mathcal E^\mathrm{Int}_1$ again following \eqref{eint1}. To deal with the term with $W$ we integrate by parts and then use the fact that $V$ is constant along $\partial\D^2$ and Proposition \ref{estZ01}:
\begin{eqnarray*}
2K(\xi)\int\limits_{\D^2}We^V&=&-\int\limits_{\D^2}V\Delta W+\int\limits_{\partial\D^2}(V\partial_\nu W-W\partial_\nu V)\\
&=&V(\xi)\int\limits_{\partial\D^2}\partial_\nu W-2h(\xi)\frac{2\varphi(\xi)}{\varphi(\xi)^2+K(\xi)}\int\limits_{\partial\D^2}W\\
&=&\frac{2\varphi(\xi)h(\xi)}{\varphi(\xi)^2+K(\xi)}\int\limits_{\partial\D^2}\(-\frac2\pi \frac{2\varphi(\xi)}{\varphi(\xi)^2+K(\xi)}\int\limits_{\partial\D^2}\log|z-w|(h(f(w))-h(\xi))dw\)dz\\
&=&-\frac{8\varphi(\xi)^2h(\xi)}{\pi\(\varphi(\xi)^2+K(\xi)\)^2}\int\limits_{\partial\D^2}(h(f(w))-h(\xi))\(\int\limits_{\partial\D^2}\log|z-w|dz\)dw\\
&=&0.
\end{eqnarray*}
Applying Proposition \ref{estZ01} again gives
$$2K(\xi)\tau \int\limits_{\D^2}e^V=\frac{8\pi K(\xi)}{\varphi(\xi)^2+K(\xi)}\tau=\frac{8\pi K(1)}{\varphi(\xi)^2+K(\xi)}\tau+O(|\eta||\tau|),$$
therefore
\begin{eqnarray}
\nonumber\int\limits_{\D^2}\mathcal E^\mathrm{Int}_1&=&2K(\xi)\int\limits_{\D^2}We^V+2K(\xi)\tau \int\limits_{\D^2}e^V+\int\limits_{\D^2}O\(|f(z)-\xi|(|W|+|\tau|)+(|W|+|\tau|)^2\(1+e^{W+\tau}\)\)\\
\label{eint1z0}&=&\frac{8\pi K(1)}{\varphi(\xi)^2+K(\xi)}\tau+O\(\delta^2+|\tau|^2\).
\end{eqnarray}
Concerning the second term, we write
\begin{eqnarray*}
\mathcal E^\mathrm{Int}_2&=&2\left\langle\nabla K(\xi),f(z)-\xi\right\rangle e^V+O\(|f(z)-\xi|^2+|\e||f(z)-\xi|\)\\
&=&2\delta\left\langle\nabla K(\xi),\frac{z-\xi}{1+(1-\delta)\overline\xi z}\right\rangle e^V+O\(\frac{\delta^2}{(\delta+|z+\xi|)^2}+\frac{|\e|\delta}{\delta+|z+\xi|}\);
\end{eqnarray*}
then, we argue as in the proof of \eqref{EintZ1}, \eqref{EintZ2} and apply Proposition \ref{wZ0}:
\begin{eqnarray*}
\int\limits_{\D^2}\left\langle\nabla K(\xi),\frac{z-\xi}{1+(1-\delta)\overline\xi z}\right\rangle e^V
&=&\langle\nabla K(\xi),\xi\rangle\int\limits_{\D^2}\frac{(1-\delta)^2|w|^2-1+\delta w_1}{|1+(1-\delta)w|^2}\frac{4\varphi(\xi)^2}{\(\varphi(\xi)^2+K(\xi)|w|^2\)^2}dw\\
&+&\left\langle\nabla K(\xi),\xi^\perp\right\rangle\int\limits_{\D^2}\frac{(2-\delta)w_2}{|1+(1-\delta)w|^2}\frac{4\varphi(\xi)^2}{\(\varphi(\xi)^2+K(\xi)|w|^2\)^2}dw\\
&=&-\frac{4\pi}{\varphi(\xi)^2+K(\xi)}\langle\nabla K(\xi),\xi\rangle\\
&=&-\frac{4\pi}{\varphi(\xi)^2+K(\xi)}\partial_1K(1)+O(|\eta|).
\end{eqnarray*}
Therefore,
\begin{eqnarray}
\!\!\!\!\!\!\!\!\!\!\nonumber\int\limits_{\D^2}\mathcal E^\mathrm{Int}_2&=&2\delta\int\limits_{\D^2}\left\langle\nabla K(\xi),\frac{z-\xi}{1+(1-\delta)\overline\xi z}\right\rangle e^V+\int\limits_{\D^2}O\(\frac{\delta^2}{(\delta+|z+\xi|)^2}+\frac{\e\delta}{\delta+|z+\xi|}\)\\
\!\!\!\!\!\!\!\!\!\!\label{eint2z0}&=&-\frac{8\pi\partial_1 K(1)}{\varphi(\xi)^2+K(\xi)}\delta+O\(\delta\(\delta\log\frac1\delta+|\eta|+|\e|\)\).
\end{eqnarray}
Finally, since $G(\xi)=O(|\eta|)$, then
$$\int\limits_{\D^2}\mathcal E^\mathrm{Int}_3=O(|\eta||\e|),$$
and the result follows by \eqref{eint1z0} and \eqref{eint2z0}.
\end{proof}\

\begin{proposition}\label{ebdz0}
$$\int\limits_{\partial\D^2}\mathcal E^\partial=\frac{4\pi\varphi(\xi)}{\varphi(\xi)^2+K(\xi)}\(-2\delta(-\Delta)^\frac12h(1)+\tau h(1)\)+O\(\frac{|\e|^2}{\log\frac1{|\e|}}\).$$
\end{proposition}

\begin{proof}
We split again as in \eqref{split2}, and we write
\begin{eqnarray*}
\!\!\!\!\!\!\!\!\!\!\mathcal E^\partial_1&=&\frac{h(\xi)2\varphi(\xi)}{\varphi(\xi)^2+K(\xi)}W+\frac{h(\xi)2\varphi(\xi)}{\varphi(\xi)^2+K(\xi)}\tau \\
&+&O(|f(z)-\xi|(|W|+\tau))+O\((|W|+|\tau|)^2\(1+e^{W+\tau}\)\).
\end{eqnarray*}
The integral with the first term vanishes because, as in the proof of Proposition \ref{eintz0},
$$\!\!\!\!\!\!\!\!\!\!\frac{h(\xi)2\varphi(\xi)}{\varphi(\xi)^2+K(\xi)}\int\limits_{\partial\D^2}W=-\frac{4\varphi(\xi)^2h(\xi)}{\pi\(\varphi(\xi)^2+K(\xi)\)^2}\int\limits_{\partial\D^2}(h(f(w))-h(\xi))\(\int\limits_{\partial\D^2}\log|z-w|dz\)dw=0.$$
On the other hand, the second term is constant, 
$$\frac{h(\xi)2\varphi(\xi)}{\varphi(\xi)^2+K(\xi)}\tau \int\limits_{\partial\D^2}1=\frac{4\pi\varphi(\xi)}{\varphi(\xi)^2+K(\xi)}\tau h(\xi)=\frac{4\pi\varphi(\xi)}{\varphi(\xi)^2+K(\xi)}\tau h(1)+O(|\eta||\tau|);$$
and therefore,
\begin{eqnarray*}
\int\limits_{\partial\D^2}\mathcal E^\partial_1&=&\frac{4\pi\varphi(\xi)}{\varphi(\xi)^2+K(\xi)}\tau h(1)+O(|\eta||\tau|)+\int\limits_{\partial\D^2}\frac\delta{\delta+|z+\xi|}\delta\(1+\log\frac1{|z+\xi|}\)\\
&+&\int\limits_{\partial\D^2}\(\delta\(1+\log\frac1{|z+\xi|}\)+|\tau|\)^2\(1+\frac1{|z+\xi|^{O(\delta)}}\)\\
&=&\frac{4\pi\varphi(\xi)}{\varphi(\xi)^2+K(\xi)}\tau h(1)+O\(\delta^2\log^2\frac1\delta+|\eta||\tau|+|\tau|^2\).
\end{eqnarray*}
To estimate $\mathcal E^\partial_2$, \eqref{gi1} yields $I(f(z))=O(|f(z)-\xi|+|\eta|)$, hence
$$\int\limits_{\partial\D^2}\mathcal E^\partial_2=\int\limits_{\partial\D^2}O\(|\e|(|f(z)-\xi|+|\eta|)e^W\)=O\(\e\(\delta\log\frac1\delta+|\eta|\)\).$$
Finally, since $\mathcal E^\partial_3$ is constant, we can use Proposition \ref{intFracLap} to conclude the proof:
\begin{eqnarray*}
\int\limits_{\partial\D^2}\mathcal E^\partial_3&=&-\frac{8\pi\varphi(\xi)}{\varphi(\xi)^2+K(\xi)}\delta (-\Delta)^\frac12h(\xi)+O\(\delta^2\)\\
&=&-\frac{8\pi\varphi(\xi)}{\varphi(\xi)^2+K(\xi)}\delta (-\Delta)^\frac12h(1)+O(\delta(\delta+|\eta|)).
\end{eqnarray*}
\end{proof}\

We can finally write the exact expressions of the constants $c_0,c_1$ and $c_2$.

\begin{corollary}\label{mainorder}
The constants $c_0,c_1,c_2$ in problem \eqref{projProb} satisfy:
\begin{eqnarray*}
c_1&=&-A_0\frac4{\(\varphi(1)^2+K(1)\)^2}\delta\(\eta\underbrace{\(\partial_{12}K(1)+2\varphi(1)(-\Delta)^\frac12h'(1)\)}_{=:\mathfrak a_{11}}+\e\underbrace{\(\partial_1G(1)+2\varphi(1)(-\Delta)^\frac12I(1)\)}_{=:\mathfrak b_1}\right.\\
&+&\left.\delta\log\frac1\delta\underbrace{\(-\frac{2\varphi(1)^2}{\varphi(1)^2+K(1)}\(\Delta K(1)+4|\nabla H(1)|^2\)\)}_{=:\mathfrak a_{12}}\)+O\(\frac{|\e|^2}{\log^2\frac1{|\e|}}+\|\phi\|^2e^{O\(\|\phi\|^2\)}\),\\
c_2&=&-A_0\frac4{\(\varphi(1)^2+K(1)\)^2}\delta\(\eta\underbrace{\(\partial_{22}K(1)+2\varphi(1)h''(1)\)}_{=:\mathfrak a_{21}}+\e\underbrace{\(\partial_2G(1)+2\varphi(1)I'(1)\)}_{=:\mathfrak b_2}\)\\
&+&O\(\frac{|\e|^2}{\log^2\frac1{|\e|}}+\|\phi\|^2e^{O\(\|\phi\|^2\)}\),\\
c_0&=&-\frac1{3\pi}\frac4{\varphi(1)^2+K(1)}\(\delta\underbrace{\(-\(2\partial_1K(1)+2\varphi(1)(-\Delta)^\frac12h(1)\)\)}_{=:\mathfrak a_{32}}+\tau\underbrace{(2K(1)+\varphi(1)h(1))}_{=:\mathfrak a_{33}}\)\\
&+&O\(\frac{|\e|^2}{\log\frac1{|\e|}}+\|\phi\|^2e^{O\(\|\phi\|^2\)}\),
\end{eqnarray*}
where
$$A_0:=\frac{3\varphi(1)^2\(\varphi(1)^2+K(1)\)^3}{\(3\varphi(1)^4+3\varphi(1)^2K(1)+2K(1)^2\)}.$$
\end{corollary}

\begin{proof}
We only prove the estimate for $c_1$, since the others are similar.

We first recall that $\xi=1$ is a critical point of $\varphi$ given in \eqref{phi}, which implies
$$\partial_1K(1)+2\varphi(1)(-\Delta)^\frac12h(1)=0,\quad\quad\quad\varphi(\xi)=\varphi(1)+O\(|\eta|^2\);$$
therefore, from \eqref{EintZ1} and \eqref{EbdrZ1} we get
\begin{eqnarray}
\nonumber&&\int\limits_{\D^2}\mathcal E^\mathrm{Int}\mathcal Z_1+\int\limits_{\partial\D^2}\mathcal E^\partial\mathcal Z_1\\
\nonumber&=&\frac{4\pi}{\(\varphi(\xi)^2+K(\xi)\)^2}\delta\(\(\partial_1K(1)+\eta\partial_{12}K(1)+\e\partial_1G(1)-\delta\log\frac1\delta\frac{2\varphi(1)^2}{\varphi(1)^2+K(1)}\Delta K(1)\)\right.\\
\nonumber&+&2\varphi(\xi)\((-\Delta)^\frac12h(1)+\eta(-\Delta)^\frac12h'(1)+\e(-\Delta)^\frac12I(1)-\left.\delta\log\frac1\delta\frac{4\varphi(1)}{\varphi(1)^2+K(1)}\((-\Delta)^\frac12h(1)^2+h'(1)^2\)\)\)\\
\nonumber&+&O\(\frac{|\e|^2}{\log\frac1{|\e|}}\)\\
\nonumber&=&\frac{4\pi}{\(\varphi(1)^2+K(1)\)^2}\delta\(\eta\(\partial_{12}K(1)+2\varphi(1)(-\Delta)^\frac12h'(1)\)+\e\(\partial_1G(1)+2\varphi(1)(-\Delta)^\frac12I(1)\)\right.\\
\label{intez1}&-&\left.\delta\log\frac1\delta\frac{2\varphi(1)}{\varphi(1)^2+K(1)}\(\Delta K(1)+|\nabla H(1)|^2\)\)+O\(\frac{|\e|^2}{\log\frac1{|\e|}}\).
\end{eqnarray}
Moreover, Lemmas \ref{l} and \ref{n} give, for $1<p<\frac43$,
\begin{eqnarray}
\nonumber&&\int\limits_{\D^2}\mathcal L^\mathrm{Int}\phi\mathcal Z_1+\int\limits_{\partial\D^2}\mathcal L^\partial\phi\mathcal Z_1+\int\limits_{\D^2}\mathcal N^\mathrm{Int}(\phi)\mathcal Z_1+\int\limits_{\partial\D^2}\mathcal N^\partial(\phi)\mathcal Z_1\\
\nonumber&=&O\(\left\|\mathcal L^\mathrm{Int}\phi\right\|_{L^p\(\D^2\)}+\left\|\mathcal L^\partial\phi\right\|_{L^p\(\partial\D^2\)}+\left\|\mathcal N^\mathrm{Int}(\phi)\right\|_{L^p\(\D^2\)}+\left\|\mathcal N^\mathrm{Int}(\phi)\right\|_{L^p\(\partial\D^2\)}\)\\
\nonumber&=&O\(\frac{|\e|}{\log\frac1{|\e|}}\|\phi\|+\|\phi\|^2e^{O\(\|\phi\|^2\)}\)\\
\label{intln}&=&O\(\frac{|\e|^2}{\log^2\frac1{|\e|}}+\|\phi\|^2e^{O\(\|\phi\|^2\)}\),
\end{eqnarray}
where in the last inequality we have applied \eqref{phio}. Finally, from Proposition \ref{z1z2} we have
$$2K(\xi)\int\limits_{\D^2}e^V\mathcal Z_1^2+\frac{h(\xi)2\varphi(\xi)}{\varphi(\xi)^2+K(\xi)}\int\limits_{\partial\D^2}\mathcal Z_1^2=\frac\pi{A_0}+O(|\eta|),$$
and the conclusion follows by putting this together with \eqref{intez1} and \eqref{intln}.
\end{proof}

\medskip

\section{The finite dimensional reduction: proof of Theorem \ref{main}}\label{6}

Take $\delta,|\eta|,|\tau|,|\e|$ small enough so that Proposition \ref{contr} can be applied to find a solution to \eqref{projProb}. If we have $c_0=c_1=c_2=0$ for some $\delta,\eta,\tau$, then $\phi$ also solves \eqref{eqphi}, hence we get a solution to \eqref{p}. From fary \ref{mainorder} and the estimates on $\|\phi\|$,
assuming that
$$\eta=s\e,\;\tau=t\e\(\log\frac1\delta\)^{-1}\;\hbox{and}\;\delta\log\frac1\delta=d\e\;\hbox{with}\;d>0\;\hbox{if}\;\e>0\;\hbox{or}\;d<0\;\hbox{if}\;\e<0,$$
we have that all the $c_i$'s are zero if
\begin{equation}\label{sis-fin}\left\{\begin{aligned}
&\mathfrak a_{11}s+\mathfrak a_{12}d+\mathfrak b_1+o_\e(1)=0,\\
&\mathfrak a_{21}s+\mathfrak b_2+o_\e(1)=0,\\
&\mathfrak a_{32}d+\mathfrak a_{33}t+o_\e(1)=0.
\end{aligned}\right.\end{equation}
The coefficients $\mathfrak a_{ij}$'s are defined in Corollary \ref{mainorder}.
System \eqref{sis-fin} can be rewritten as $\mathfrak F_\e(s,d,t)=\mathfrak F_0(s,d,t)+o_\e(1)=0$ where the function $\mathfrak F_0:\R^3\to\R^3$ is defined by
$$\mathfrak F_0(s,d,t):=\mathfrak A\(\begin{matrix}s\\d\\t\\\end{matrix}\)+\mathfrak B,\;\hbox{with}\;\mathfrak A:=\(\begin{matrix}\mathfrak a_{11}&\mathfrak a_{12}&0\\\mathfrak a_{21}&0&0\\0&\mathfrak a_{32}&\mathfrak a_{33}\\\end{matrix}\)\;\hbox{and}\;\mathfrak B:=\(\begin{matrix}\mathfrak b_1\\\mathfrak b_2\\0\\\end{matrix}\).$$
Now it is clear that if 
$$\det\mathfrak A=\det\(\begin{matrix}\mathfrak a_{11}&\mathfrak a_{12}&0\\
\mathfrak a_{21}&0&0\\0&\mathfrak a_{32}&\mathfrak a_{33}\\\end{matrix}\)\not=0,\quad\mbox{i.e.,}\quad\mathfrak a_{12}\mathfrak a_{21}\mathfrak a_{33}\not=0\;\hbox{(see \eqref{h2k}, \eqref{cond10}, \eqref{cond100} and \eqref{cond1})},$$ 
there exists a unique $(s_0,d_0,t_0)\in\R^3$ such that $\mathfrak F_0(s_0,d_0,t_0)=0$ with $d_0\not=0$ if
$$\det\(\begin{matrix}\mathfrak a_{11}&\mathfrak b_1&0\\\mathfrak a_{21}&\mathfrak b_2&0\\0&0&\mathfrak a_{33}\\\end{matrix}\)\not=0,\quad\hbox{i.e.,}\quad\mathfrak a_{33}\(\mathfrak a_{11}\mathfrak b_2-\mathfrak a_{21}\mathfrak b_1\)\not=0\;\hbox{(see \eqref{cond1000})}.$$ 
Moreover, the Brouwer degree of $\mathfrak F_\e$ is not zero and since $\mathfrak F_\e\to\mathfrak F_0$ uniformly on compact sets of $\R\times\R\setminus\{0\}\times\R$, there exists $(s_\e,d_\e,t_\e)\in\R\times\R\setminus\{0\}\times\R$ such that $\mathfrak F_\e(s_\e,d_\e,t_\e)=0$ with $(s_\e,d_\e,t_\e)\to(s_0,d_0,t_0)$ as $\e\to0$. That concludes the proof of the existence of the solution.

Thanks to the estimates on $\mathcal E,\mathcal L,\mathcal N$ and $\|\phi\|$, from \eqref{eqphi} we get $\|\Delta\phi\|_{L^p\(\D^2\)}+\left\|\partial_\nu\phi\right\|_{L^p\(\partial\D^2\)}=o_\e(1)$ for some $p>1$, therefore $\|\phi\|_{L^\infty\(\D^2\)}=o_\e(1)$. Moreover, since $V\(f^{-1}(z)\),W\(f^{-1}(z)\)$ both concentrate at $\xi=1$, we conclude that the solution
$$u=V\(f^{-1}(z)\)+W\(f^{-1}(z)\)+\tau+\phi\(f^{-1}(z)\)$$
concentrates at $\xi=1$. This finishes the proof of Theorem \ref{main}.

\medskip

\section{Appendix: some useful computations}

In this section we collect some useful computations required for the estimates in Section 2 and Section 5.

\begin{proposition}\label{f}
Let $\xi\in\partial\D^2$. For any $z\in\D^2$ one has
\begin{equation}\label{fxi}
|f(z)-\xi|=O\(\frac\delta{\delta+|z+\xi|}\),
\end{equation}
and in particular
$$\|f(z)-\xi\|_{L^p\(\D^2\)}=\begin{cases}O(\delta)&1\le p<2,\\O\(\delta\sqrt{\log\frac1\delta}\)&p=2,\\O\(\delta^\frac2p\)&p>2,\end{cases}\quad\quad\|f(z)-\xi\|_{L^p\(\partial\D^2\)}=\begin{cases}O\(\delta\log\frac1\delta\)&p=1,\\O\(\delta^\frac1p\)&p>1.\end{cases}$$
Moreover, if $z\in\partial\D^2$ and $h\in C^2\(\partial\D^2\)$, then
$$h(f(z))-h(\xi)=\delta h'(\xi)\Theta(z)+O\(\frac{\delta^2}{(\delta+|z+\xi|)^2}\),$$
with
$$\Theta(z)=\Theta_{\delta,\xi}(z):=\frac{2\left\langle z,\xi^\perp\right\rangle}{1+(1-\delta)^2+2(1-\delta)\langle z,\xi\rangle}.$$
\end{proposition}

\begin{proof}
From the definition of $f(z)$, we have
$$|f(z)-\xi|=\left|\frac{\delta(z-\xi)}{1+(1-\delta)\overline\xi z}\right|=\frac{\delta|z-\xi|}{|(1-\delta)z+\xi|}=O\(\frac\delta{|(1-\delta)z+\xi|}\).$$
Moreover, since $z\in\D^2$ and $\xi\in\partial\D^2$, then
\begin{eqnarray}
\nonumber|(1-\delta)z+\xi|^2&=&\delta^2+(1-\delta)|z+\xi|^2+\delta(1-\delta)\(1-|z|^2\)\\
\nonumber&\ge&\delta^2+(1-\delta)|z+\xi|^2\\
\nonumber&\ge&\frac{\delta^2+|z+\xi|^2}2\\
\label{deltazxi}&\ge&\frac{(\delta+|z+\xi|)^2}4
\end{eqnarray}
hence $\delta+|z+\xi|=O\(|(1-\delta)z+\xi|\)$, which proves \eqref{fxi}.

Take now $z\in\partial\D^2$ and $h\in C^2\(\partial\D^2\)$; by the chain rule we have
\begin{equation}\label{htheta}
h(f(z))-h(\xi)=h'(\xi)\left\langle f(z)-\xi,\xi^\perp\right\rangle+O\(|f(z)-\xi|^2\).
\end{equation}
Moreover, the scalar product equals
\begin{eqnarray*}
\left\langle f(z)-\xi,\xi^\perp\right\rangle&=&\mathrm{Im}\(\overline\xi(f(z)-\xi)\)\\
&=&\mathrm{Im}{\frac{\delta\(\overline\xi z-1\)}{1+(1-\delta)\overline\xi z}}\\
&=&\delta\mathrm{Im}\frac{\(\overline\xi z-1\)\(1+(1-\delta)\xi\overline z\)}{\left|1+(1-\delta)\overline\xi z\right|^2}\\
&=&\delta\mathrm{Im}\(\frac{\delta\langle z,\xi\rangle+\imath(2-\delta)\left\langle z,\xi^\perp\right\rangle}{1+(1-\delta)^2+2(1-\delta)\langle z,\xi\rangle}\)\\
&=&\delta\(1-\frac\delta2\)\Theta(z).
\end{eqnarray*}
The conclusion follows by putting together this expression with \eqref{fxi} and \eqref{htheta}.
\end{proof}

\begin{proposition}\label{intFracLap}
Given $\xi\in\partial\D^2$,
$$\int\limits_{\partial\D^2}(h(f(z))-h(\xi))dz=-2\pi\delta(-\Delta)^\frac12h(\xi)+O\(\delta^2\).$$
\end{proposition}

\begin{proof}
Making the change of variables $w=f(z)$ we get that
\begin{eqnarray}
\nonumber\int\limits_{\partial\D^2}(h(f(z))-h(\xi))dz&=&\int\limits_{\partial\D^2}(h(w)-h(\xi))\left|\(f^{-1}\)'(w)\right|dw\\
\nonumber&=&\int\limits_{\partial\D^2}(h(w)-h(\xi))\frac{1-(1-\delta)^2}{\left|1-(1-\delta)\overline\xi w\right|^2}dw\\
\label{A0}&=&\(2\delta+O\(\delta^2\)\)\int\limits_{\partial\D^2}\frac{h(w)-h(\xi)-h'(\xi)\left\langle w,\xi^\perp\right\rangle}{\left|1-(1-\delta)\overline\xi w\right|^2}dw,
\end{eqnarray}
where we have used that $\left|1-(1-\delta)\overline\xi w\right|^2$ is even with respect to $\left\langle w,\xi^\perp\right\rangle$.
Since
$$\frac1{\left|1-(1-\delta)\overline\xi w\right|^2}=\frac1{\delta^2+(1-\delta)|w-\xi|^2}=\begin{cases}O\(\frac1{\delta^2}\)&\hbox{if }|w-\xi|\le\delta,\\
\frac1{|w-\xi|^2}\(1+O\(\delta+\frac{\delta^2}{|w-\xi|^2}\)\)&\hbox{if }|w-\xi|>\delta,\end{cases}$$
then
\begin{equation}\label{A2}
\!\!\!\!\!\!\!\!\!\!\int\limits_{\{|w-\xi|\le\delta\}}\frac{h(w)-h(\xi)-h'(\xi)\left\langle w,\xi^\perp\right\rangle}{\left|1-(1-\delta)\overline\xi w\right|^2}dw=\int\limits_{\{|w-\xi|\le\delta\}}O\(\frac{|w-\xi|^2}{\delta^2}\)dw=O(\delta).
\end{equation}
Likewise,
\begin{eqnarray}
\!\!\!\!\!\!\!\!\!\!\!\!\!\!\!\!\!\!\!\!\nonumber&&\int\limits_{\{|w-\xi|>\delta\}}\frac{h(w)-h(\xi)-h'(\xi)\left\langle w,\xi^\perp\right\rangle}{\left|1-(1-\delta)\overline\xi w\right|^2}dw\\
\!\!\!\!\!\!\!\!\!\!\!\!\!\!\!\!\!\!\!\!\nonumber&=&\(\int\limits_{\{|w-\xi|>\delta\}}\frac{h(w)-h(\xi)}{|w-\xi|^2}dw\)(1+O(\delta))+\int\limits_{\{|w-\xi|>\delta\}}O\(\frac{\delta^2}{|w-\xi|^2}\)dw\\
\!\!\!\!\!\!\!\!\!\!\!\!\!\!\!\!\!\!\!\!\label{A1}&=&\int\limits_{\{|w-\xi|>\delta\}}\frac{h(w)-h(\xi)}{|w-\xi|^2}dw+O(\delta);
\end{eqnarray}
and recalling that, by the definition of the fractional Laplacian,
$$(-\Delta)^\frac12h(\xi)=-\frac1\pi\int\limits_{\{|w-\xi|>\delta\}}\frac{h(w)-h(\xi)}{|w-\xi|^2}dw+O(\delta),$$
the conclusion follows putting \eqref{A0}, \eqref{A2} and \eqref{A1} together.
\end{proof}

\begin{proposition}\label{estKgrad}The following identities hold true:
\begin{enumerate}
\item $\int\limits_{\D^2}\frac{(1-\delta)|w|^2-1+\delta w_1}{|1+(1-\delta)w|^2}\frac{4\varphi(\xi)^2}{\(\varphi(\xi)^2+K(\xi)|w|^2\)^3}w_1dw=\frac{2\pi}{\(\varphi(\xi)^2+K(\xi)\)^2}+O(\delta)$.
\item $\int\limits_{\D^2}\frac{(2-\delta)w_2}{|1+(1-\delta)+w|^2}\frac{4\varphi(\xi)^2}{\(\varphi(\xi)^2+K(\xi)|w|^2\)^3}w_1dw=0$.
\item $\int\limits_{\D^2}\frac{(1-\delta)|w|^2-1+\delta w_1}{|1+(1-\delta)w|^2}\frac{4\varphi(\xi)^2}{\(\varphi(\xi)^2+K(\xi)|w|^2\)^3}w_2dw=0$.
\item $\int\limits_{\D^2}\frac{(2-\delta)w_2}{|1+(1-\delta)w|^2}\frac{4\varphi(\xi)^2}{\(\varphi(\xi)^2+K(\xi)|w|^2\)^3}w_2dw=\frac{2\pi}{\(\varphi(\xi)^2+K(\xi)\)^2}+O(\delta)$.
\end{enumerate}
\end{proposition}

\begin{proof}
We decompose
$$\frac{(1-\delta)|w|^2-1+\delta w_1}{|1+(1-\delta)w|^2}=\frac1{1-\delta}-(2-\delta)\frac{1+(1-\delta)w_1}{|1+(1-\delta)w|^2};$$
since, by symmetry with respect to $w_1$,
$$\int\limits_{\D^2}\frac{4\varphi(\xi)^2}{\(\varphi(\xi)^2+K(\xi)|w|^2\)^3}w_1dw=0,$$
then we suffice to show
$$\int\limits_{\D^2}\frac{1+(1-\delta)w_1}{|1+(1-\delta)w|^2}\frac{4\varphi(\xi)^2}{\(\varphi(\xi)^2+K(\xi)|w|^2\)^3}w_1dw=-\frac\pi{\(\varphi(\xi)^2+K(\xi)\)^2}+O(\delta).$$
By using polar coordinates we get:
\begin{eqnarray*}
&&\int\limits_{\D^2}\frac{1+(1-\delta)w_1}{|1+(1-\delta)w|^2}\frac{4\varphi(\xi)^2}{\(\varphi(\xi)^2+K(\xi)|w|^2\)^3}w_1dw\\
&=&\int\limits_0^1\(\int\limits_{-\pi}^\pi r\frac{1+(1-\delta)r\cos t}{1+2(1-\delta)r\cos t+(1-\delta)^2r^2}\frac{4\varphi(\xi)^2}{\(\varphi(\xi)^2+K(\xi)r^2\)^3}r\cos tdt\)dr\\
&=&\varphi(\xi)^2\int\limits_0^1\frac{2r^2}{\(\varphi(\xi)^2+K(\xi)r^2\)^3}\(\int\limits_{-\pi}^\pi\(\cos t-\frac{(1-\delta)^2r^2-1}{2(1-\delta)r}-\frac{\frac{1-(1-\delta)^4r^4}{4(1-\delta)^2r^2}}{\frac{1+(1-\delta)^2r^2}{2(1-\delta)r}+\cos t}\)dt\)dr\\
&=&\varphi(\xi)^2\int\limits_0^1\frac{2r^2}{\(\varphi(\xi)^2+K(\xi)r^2\)^3}2\pi\(-\frac{(1-\delta)^2r^2-1}{2(1-\delta)r}-\frac{\frac{1-(1-\delta)^4r^4}{4(1-\delta)^2r^2}}{\sqrt{\(\frac{1+(1-\delta)^2r^2}{2(1-\delta)r}\)^2-1}}\)dr\\
&=&-\pi(1-\delta)\varphi(\xi)^2\int\limits_0^1\frac{4r^3}{K(\xi)\(\varphi(\xi)^2+K(\xi)r^2\)^2}dr\\
&=&-(1-\delta)\frac\pi{\(\varphi(\xi)^2+K(\xi)\)^2},
\end{eqnarray*}
where we used
\begin{equation}\label{resid}
\int\limits_{-\pi}^\pi\frac{dt}{a+\cos t}=\frac{2\pi}{\sqrt{a^2-1}}\quad\hbox{for }|a|>1.
\end{equation}
This proves (1).

Since the integrands are odd in $w_2$, (2) and (3) follow straightforward. Finally, (4) can be proved similarly exploiting again \eqref{resid}:
\begin{eqnarray*}
&&\int\limits_{\D^2}\frac{w_2}{|1+(1-\delta)w|^2}\frac{4\varphi(\xi)^2}{\(\varphi(\xi)^2+K(\xi)|w|^2\)^3}w_2dw\\
&=&\int\limits_0^1\(\int\limits_{-\pi}^\pi r\frac{r\sin t}{1+2(1-\delta)r\cos t+(1-\delta)^2r^2}\frac{4\varphi(\xi)^2}{\(\varphi(\xi)^2+K(\xi)r^2\)^3}r\sin tdt\)dr\\
&=&\frac{\varphi(\xi)^2}{1-\delta}\int\limits_0^1\frac{2r^2}{\(\varphi(\xi)^2+K(\xi)r^2\)^3}\(\int\limits_{-\pi}^\pi\(\frac{1+(1-\delta)^2r^2}{2(1-\delta)r}-\cos t-\frac{\(\frac{1-(1-\delta)^2r^2}{2(1-\delta)r}\)^2}{\frac{1+(1-\delta)^2r^2}{2(1-\delta)r}+\cos t}\)dt\)dr\\
&=&\frac{\varphi(\xi)^2}{1-\delta}\int\limits_0^1\frac{2r^2}{\(\varphi(\xi)^2+K(\xi)r^2\)^3}2\pi\(\frac{1+(1-\delta)^2r^2}{2(1-\delta)r}-\frac{\(\frac{1-(1-\delta)^2r^2}{2(1-\delta)r}\)^2}{\sqrt{\(\frac{1+(1-\delta)^2r^2}{2(1-\delta)r}\)^2-1}}\)dr\\
&=&\pi\varphi(\xi)^2\int\limits_0^1\frac{4r^3}{\(\varphi(\xi)^2+K(\xi)r^2\)^3}dr\\
&=&\frac\pi{\(\varphi(\xi)^2+K(\xi)\)^2},
\end{eqnarray*}
and the result follows by multiplying by $2-\delta$.
\end{proof}

\begin{proposition}\label{estKlap}The following identities hold true:
\begin{enumerate}
\item $\int\limits_{\Omega_\delta}\frac{dw}{(1+w_1)^2+w_2^2}=\pi\log\frac1\delta+O(1)$,
\item $\int\limits_{\Omega_\delta}\frac{(1+w_1)^2-w_2^2}{\((1+w_1)^2+w_2^2\)^2}dw=O(1)$,
\item $\int\limits_{\Omega_\delta}\frac{(1+w_1)w_2}{\((1+w_1)^2+w_2^2\)^2}dw=0$,
\end{enumerate}
where 
$\Omega_\delta:=\left\{\left|w-\frac{1-\delta}\delta\right|\le\frac{1-\delta}\delta\right\}$.
\end{proposition}

\begin{proof}
Changing to polar coordinates
$$\int\limits_{\Omega_\delta}\frac{dw}{(1+w_1)^2+w_2^2}=\int\limits_{\Omega_\delta}\frac r{1+2r\cos t+r^2}drdt,\quad \Omega_\delta=\left\{r\le\frac{2(1-\delta)}\delta\cos t,\;-\frac\pi2\le t\le\frac\pi2\right\}.$$
Since $\cos t>0$ on $\Omega_\delta$,
$$\frac r{1+2r\cos t+r^2}-\frac1{1+r}=O\(\frac1{(1+r)^2}\)\quad\hbox{uniformly in }t,$$
and we have 
\begin{eqnarray*}
\int\limits_{\Omega_\delta}\frac r{1+2r\cos t+r^2}drdt&=&\int\limits_{\Omega_\delta}\frac{drdt}{1+r}+O(1)\\
&=&\int\limits_{-\frac\pi2}^\frac\pi2\log\(1+\frac{2(1-\delta)}\delta\cos t\)dt+O(1)\\
&=&\log\frac1\delta\int\limits_{-\frac\pi2}^\frac\pi2dt+\int\limits_{-\frac\pi2}^\frac\pi2\log(\delta+2(1-\delta)\cos t)dt+O(1)\\
&=&\pi\log\frac1\delta+O(1),
\end{eqnarray*}
and (1) is proved.

Since $(1+w_1)^2-w_2^2=w_1^2-w_2^2+O(|w|)$ then
\begin{eqnarray*}
\int\limits_{\Omega_\delta}\frac{(1+w_1)^2-w_2^2}{\((1+w_1)^2+w_2^2\)^2}dw&=&\int\limits_{\Omega_\delta}\frac{w_1^2-w_2^2}{\((1+w_1)^2+w_2^2\)^2}dw+O\(\int\limits_{\{w_2\ge0\}}\frac{|w|}{\((1+w_1)^2+w_2^2\)^2}dw\)\\
\!\!\!\!\!\!\!\!\!\!&=&\int\limits_{\Omega_\delta}\frac{r^3\cos(2t)}{\(1+2r\cos t+r^2\)^2}drdt+O(1).
\end{eqnarray*}
As before,
$$\frac{r^3}{\(1+2r\cos t+r^2\)^2}-\frac1{1+r}=O\(\frac1{(1+r)^2}\)\quad\hbox{uniformly in }t,$$
leading to
$$\!\!\!\!\!\!\!\!\!\!\int\limits_{\Omega_\delta}\frac{r^3\cos(2t)}{\(1+2r\cos t+r^2\)^2}drdt=\int\limits_{\Omega_\delta}\frac{\cos(2t)}{1+r}drdt+O\(\int\limits_{\{w_2\ge0\}}\frac{drdt}{(1+r)^2}\)=\int\limits_{\Omega_\delta}\frac{\cos(2t)}{1+r}drdt+O(1).$$
Furthermore,
\begin{eqnarray*}
\int\limits_{\Omega_\delta}\frac{\cos(2t)}{1+r}drdt&=&\int\limits_{-\frac\pi2}^\frac\pi2\cos(2t)\(\int\limits_0^{\frac{2(1-\delta)}\delta\cos t}\frac{dr}{1+r}\)dt\\
&=&\int\limits_{-\frac\pi2}^\frac\pi2\cos(2t)\log\(1+\frac{2(1-\delta)}\delta\cos t\)dt\\
&=&\int\limits_{-\frac\pi2}^\frac\pi2\cos(2t)\log(\delta+2(1-\delta)\cos t)dt\\
&=&\int\limits_{-\frac\pi2}^\frac\pi2O\(1+\log^-(2\cos t)^2\)dt\\
&=&O(1).
\end{eqnarray*}
Therefore 
$$\int\limits_{\Omega_\delta}\frac{(1+w_1)^2-w_2^2}{\((1+w_1)^2+w_2^2\)^2}dw=O(1),$$
and (2) holds.

Finally the integral in (3) vanishes since it is odd with respect to $w_2$.
\end{proof}

\begin{proposition}\label{intB} The following identities hold true:
\begin{enumerate}
\item $\int\limits_{\partial\D^2}(h(f(z))-h(\xi))\langle z,\xi\rangle dz=2\pi\delta(-\Delta)^\frac12h(\xi)+O\(\delta^2\)$,
\item $\int\limits_{\partial\D^2}(h(f(z))-h(\xi))\left\langle z,\xi^\perp\right\rangle dz=2\pi\delta h'(\xi)+O\(\delta^2\)$.
\end{enumerate}
\end{proposition}

\begin{proof}
Using Propositions \ref{intFracLap} and \ref{f}, we get
\begin{eqnarray*}
&&\int\limits_{\partial\D^2}(h(f(z))-h(\xi))\langle z,\xi\rangle dz\\
&=&\int\limits_{\partial\D^2}(h(f(z))-h(\xi))\langle z+\xi,\xi\rangle dz-\int\limits_{\partial\D^2}(h(f(z))-h(\xi))dz\\
&=&\delta h'(\xi)\int\limits_{\partial\D^2}\Theta(z)\langle z+\xi,\xi\rangle dz+\int\limits_{\partial\D^2}O\(\frac{\delta^2}{(\delta+|z+\xi|)^2}\)\langle z+\xi,\xi\rangle dz+2\pi\delta(-\Delta)^\frac12h(\xi)+O\(\delta^2\).
\end{eqnarray*}
Since $\Theta$ is odd with respect to $\xi$ and $\langle z+\xi,\xi\rangle$ is even, then the first integral vanishes; moreover, $\langle z+\xi,\xi\rangle=O\(|z+\xi|^2\)$, therefore
$$\int\limits_{\partial\D^2}O\(\frac{\delta^2}{(\delta+|z+\xi|)^2}\)\langle z+\xi,\xi\rangle dz=O\(\int\limits_{\partial\D^2}\frac{\delta^2|z+\xi|^2}{(\delta+|z+\xi|)^2}\)dz=O\(\delta^2\),$$
which proves (1).

To get (2), we use the formula
$$h(f(z))-h(\xi)=\delta(h'(\xi)+O(\delta))\Theta(z)+\frac{\delta^2h''(\xi)}2\Theta(z)^2+O\(\frac{\delta^3}{(\delta+|z+\xi|)^3}\),$$
which can be deduced by arguing as in Proposition \ref{f}. Therefore,
\begin{eqnarray*}
&&\int\limits_{\partial\D^2}(h(f(z))-h(\xi))\left\langle z,\xi^\perp\right\rangle dz\\
&=&\delta h'(\xi)\int\limits_{\partial\D^2}\Theta(z)\left\langle z,\xi^\perp\right\rangle dz+\int\limits_{\partial\D^2}O\(\delta^2\)\Theta(z)\left\langle z,\xi^\perp\right\rangle dz+\frac{\delta^2h''(\xi)}2\int\limits_{\partial\D^2}\Theta(z)^2\left\langle z,\xi^\perp\right\rangle dz\\
&+&\int\limits_{\partial\D^2}O\(\frac{\delta^3}{(\delta+|z+\xi|)^3}\)\left\langle z,\xi^\perp\right\rangle dz\\
&=&\delta h'(\xi)\int\limits_{\partial\D^2}\Theta(z)\left\langle z,\xi^\perp\right\rangle dz+\int\limits_{\partial\D^2}O\(\frac{\delta^2|z+\xi|}{\delta+|z+\xi|}\)dz+\int\limits_{\partial\D^2}O\(\frac{\delta^3|z+\xi|}{(\delta+|z+\xi|)^3}\)dz\\
&=&\delta h'(\xi)\int\limits_{\partial\D^2}\Theta(z)\left\langle z,\xi^\perp\right\rangle dz+O\(\delta^2\),
\end{eqnarray*}
where we used $\left\langle z,\xi^\perp\right\rangle=O(|z+\xi|)$ and that $\Theta(z)^2\left\langle z,\xi^\perp\right\rangle$ is odd, hence its integral vanishes. Finally,
\begin{eqnarray*}
\int\limits_{\partial\D^2}\Theta(z)\left\langle z,\xi^\perp\right\rangle dz&=&\int\limits_{\partial\D^2}\frac{2\left\langle z,\xi^\perp\right\rangle}{1+(1-\delta)^2+2(1-\delta)\langle z,\xi\rangle}\left\langle z,\xi^\perp\right\rangle dz\\
&=&\int\limits_{\partial\D^2}\(\frac{\left\langle z,\xi^\perp\right\rangle^2}{1+\langle z,\xi\rangle}+O\(\frac{\delta|z+\xi|}{\delta+|z+\xi|}\)\)dz\\
&=&\int\limits_{\partial\D^2}(1-\langle z,\xi\rangle)dz+O(\delta)\\
&=&2\pi+O(\delta),
\end{eqnarray*}
which proves (2).
\end{proof}

\begin{proposition}\label{logarctan}The following identities hold true:
\begin{enumerate}
\item $\int\limits_{\partial\D^2}\log|z+\xi|(h(f(z))-h(\xi))\langle z,\xi\rangle dz=
-2\pi\delta\log\frac1\delta(-\Delta)^\frac12h(\xi)+O(\delta)$.
\item $\int\limits_{\partial\D^2}\arctan\frac{\left\langle z,\xi^\perp\right\rangle}{1+\langle z,\xi\rangle}(h(f(z))-h(\xi))\langle z,\xi\rangle dz=-2\pi\delta\log\frac1\delta h'(\xi)+O(\delta)$,
\item $\int\limits_{\partial\D^2}\log|z+\xi|(h(f(z))-h(\xi))\left\langle z,\xi^\perp\right\rangle dz=O(\delta)$,
\item $\int\limits_{\partial\D^2}\arctan\frac{\left\langle z,\xi^\perp\right\rangle}{1+\langle z,\xi\rangle}(h(f(z))-h(\xi))\left\langle z,\xi^\perp\right\rangle dz=O(\delta)$.
\end{enumerate}
\end{proposition}

\begin{proof}
By Proposition \ref{intFracLap} we get:
\begin{eqnarray*}
&&\int\limits_{\partial\D^2}\log|z+\xi|(h(f(z))-h(\xi))\langle z,\xi\rangle dz\\
&=&\log\delta\int\limits_{\partial\D^2}(h(f(z))-h(\xi))\langle z,\xi\rangle dz+\int\limits_{\partial\D^2}\log\frac{|z+\xi|}\delta(h(f(z))-h(\xi))\langle z,\xi\rangle dz\\
&=&-2\pi\delta\log\frac1\delta(-\Delta)^\frac12h(\xi)+O\(\delta^2\log\frac1\delta\)+\int\limits_{\partial\D^2}\log\frac{|z+\xi|}\delta(h(f(z))-h(\xi))\langle z,\xi\rangle dz.
\end{eqnarray*}
To deal with the last integral, we use Proposition \ref{f} to get:
\begin{eqnarray*}\int\limits_{\partial\D^2}\log\frac{|z+\xi|}\delta(h(f(z))-h(\xi))\langle z,\xi\rangle dz&=&\delta h'(\xi)\int\limits_{\partial\D^2}\log\frac{|z+\xi|}\delta\frac{2\left\langle z,\xi^\perp\right\rangle}{1+(1-\delta)^2+2(1-\delta)\langle z,\xi\rangle}\langle z,\xi\rangle dz\\
&+&\int\limits_{\partial\D^2}\log\frac{|z+\xi|}\delta O\(\frac{\delta^2}{(\delta+|z+\xi|)^2}\)\langle z,\xi\rangle dz;
\end{eqnarray*}
the first integral vanishes due to symmetry, hence
\begin{eqnarray*}
\int\limits_{\partial\D^2}\log\frac{|z+\xi|}\delta(h(f(z))-h(\xi))\langle z,\xi\rangle dz&=&O\(\int\limits_{\partial\D^2}\left|\log\frac{|z+\xi|}\delta\right|\frac{\delta^2}{(\delta+|z+\xi|)^2}dz\)\\
&=&O\(\int_{t=O\(\frac1\delta\)}|\log t|\frac\delta{1+t^2}dt\)\\
&=&O(\delta),
\end{eqnarray*}
which proves (1).

To get (2), recalling Proposition \ref{f},
\begin{eqnarray*}
&&\int\limits_{\partial\D^2}\arctan\frac{\left\langle z,\xi^\perp\right\rangle}{1+\langle z,\xi\rangle}(h(f(z))-h(\xi))\langle z,\xi\rangle dz\\
&=&\delta h'(\xi)\int\limits_{\partial\D^2}\arctan\frac{\left\langle z,\xi^\perp\right\rangle}{1+\langle z,\xi\rangle}\Theta(z)\langle z,\xi\rangle dz+O\(\int\limits_{\partial\D^2}\arctan\frac{\left\langle z,\xi^\perp\right\rangle}{1+\langle z,\xi\rangle}\frac{\delta^2}{(\delta+|z+\xi|)^2}\langle z,\xi\rangle dz\)\\
&=&\delta h'(\xi)\int\limits_{\partial\D^2}\arctan\frac{\left\langle z,\xi^\perp\right\rangle}{1+\langle z,\xi\rangle}\frac{2\left\langle z,\xi^\perp\right\rangle}{1+(1-\delta)^2+2(1-\delta)\langle z,\xi\rangle}\langle z,\xi\rangle dz+O\(\int\limits_{\partial\D^2}\frac{\delta^2}{(\delta+|z+\xi|)^2}dz\)\\
&=&\delta h'(\xi)\int\limits_{-\pi}^\pi\frac t2\frac{2\sin t}{1+(1-\delta)^2+2(1-\delta)\cos t}\cos tdt+O(\delta)\\
&=&2\delta h'(\xi)\int\limits_0^\pi t\frac{\sin t}{1+(1-\delta)^2+2(1-\delta)\cos t}\cos tdt+O(\delta).
\end{eqnarray*}
We are left with showing that the integral equals $\pi\log\delta+O(1)$. To this purpose, we point out that
$$\frac{\sin t}{1+(1-\delta)^2+2(1-\delta)\cos t}=\frac{\pi-t}{\delta^2+(\pi-t)^2}+O(1);$$
therefore,
\begin{eqnarray*}
\int\limits_0^\pi t\frac{\sin t}{1+(1-\delta)^2+2(1-\delta)\cos t}\cos tdt&=&\int\limits_0^\pi(\pi+O(|\pi-t|))\(\frac{\pi-t}{\delta^2+(\pi-t)^2}+O(1)\)\(-1+O\(|\pi-t|^2\)\)\\
&=&-\pi\int\limits_0^\pi\frac{\pi-t}{\delta^2+(\pi-t)^2}dt+O(1)=\pi\log\delta+O(1).
\end{eqnarray*}
To deal with (3), we observe that $\left|\left\langle z,\xi^\perp\right\rangle\right|=O(|z+\xi|)$, hence
\begin{eqnarray*}
\int\limits_{\partial\D^2}\log|z+\xi|(h(f(z))-h(\xi))\left\langle z,\xi^\perp\right\rangle dz&=&\int\limits_{\partial\D^2}O\(|\log|z+\xi|||h(f(z))-h(\xi)|\left|\left\langle z,\xi^\perp\right\rangle\right|\)dz\\
&=&O\(\int\limits_{\partial\D^2}\(1+\log\frac1{|z+\xi|}\)\delta \Theta(z)|z+\xi|dz\)\\
&=&O\(\int\limits_{\partial\D^2}\(1+\log\frac1{|z+\xi|}\)\delta dz\)=O(\delta).
\end{eqnarray*}
Similarly, we can show that
$$\arctan\frac{\left\langle z,\xi^\perp\right\rangle}{1+\langle z,\xi\rangle}(h(f(z))-h(\xi))\left\langle z,\xi^\perp\right\rangle=O(\delta),$$
which easily gives (4).
\end{proof}

\begin{proposition}\label{estZ01}
The following identities hold true:
\begin{enumerate}
\item $\int\limits_{\D^2}e^{V(z)}dz=\frac{4\pi}{\varphi(\xi)^2+K(\xi)}$,
\item $\int\limits_{\partial\D^2}\log|z-w|dz=0$ for every $w\in\partial\D^2$.
\end{enumerate}
\end{proposition}

\begin{proof}
By definition
$$\int\limits_{\D^2}e^V=\int\limits_{\D^2}\frac{4\varphi(\xi)^2}{\(\varphi(\xi)^2+K(\xi)|z|^2\)^2}dz=8\pi\varphi(\xi)^2\int\limits_0^1\frac r{\(\varphi(\xi)^2+K(\xi)r^2\)^2}dr=\frac{4\pi}{\varphi(\xi)^2+K(\xi)}.$$
To get (2), we observe that for any $g\in C\(\partial\D^2\)$, the solution $W_g$ to
\begin{equation}\label{eqg}\begin{cases}-\Delta W_g=0&\hbox{in}\;\D^2\\
\partial_\nu W_g=g-\frac1{2\pi}\int\limits_{\partial\D^2}g&\hbox{on}\;\partial\D^2,\end{cases}
\end{equation}
is given by the Green's formula (see also \eqref{eqw})
$$W_g(z)=-\frac1\pi\int\limits_{\partial\D^2}\log|z-w|g(w)dw.$$
Taking $g\equiv-\pi$, the solution to \eqref{eqg} is clearly $W_{-\pi}\equiv0$, and thus
$$0=W_{-\pi}(z)=\int\limits_{\partial\D^2}\log|z-w|dw\quad\quad\quad\forall z\in\D^2.$$
\end{proof}

\begin{proposition}\label{wZ0}The following identities hold true:
\begin{enumerate}
\item $\int\limits_{\D^2}\frac{(1-\delta)|w|^2-1+\delta w_1}{|1+(1-\delta)w|^2}\frac{4\varphi(\xi)^2}{\(\varphi(\xi)^2+K(\xi)|w|^2\)^2}dw=-\frac{4\pi}{\varphi(\xi)^2+K(\xi)}$,
\item $\int\limits_{\D^2}\frac{(2-\delta)w_2}{|1+(1-\delta)w|^2}\frac{4\varphi(\xi)^2}{\(\varphi(\xi)^2+K(\xi)|w|^2\)^3}dw=0$.
\end{enumerate}
\end{proposition}

\begin{proof}
Changing to polar coordinates,
\begin{eqnarray*}
&&\int\limits_{\D^2}\frac{(1-\delta)|w|^2-1+\delta w_1}{|1+(1-\delta)w|^2}\frac{4\varphi(\xi)^2}{\(\varphi(\xi)^2+K(\xi)|w|^2\)^2}dw\\
&=&\int\limits_0^1\int\limits_{-\pi}^\pi\frac{(1-\delta)r^2-1+\delta r\cos t}{1+2(1-\delta)r\cos t+(1-\delta)^2r^2}\frac{4\varphi(\xi)^2}{\(\varphi(\xi)^2+K(\xi)r^2\)^2}rdrdt\\
&=&\frac2{1-\delta}\varphi(\xi)^2\int\limits_0^1\frac r{\(\varphi(\xi)^2+K(\xi)r^2\)^2}\(\int\limits_{-\pi}^\pi\(\delta-(2-\delta)\frac{\frac{1-(1-\delta)^2r^2}{2(1-\delta)r}}{\frac{1+(1-\delta)^2r^2}{2(1-\delta)r}+\cos t}\)dt\)dr\\
&=&\frac2{1-\delta}\varphi(\xi)^2\int\limits_0^1\frac r{\(\varphi(\xi)^2+K(\xi)r^2\)^2}4\pi(\delta-1)dr\\
&=&-\frac{4\pi}{\varphi(\xi)^2+K(\xi)},
\end{eqnarray*}
where we used \eqref{resid}.

The second identity follows by oddness with respect to $w_2$.
\end{proof}

\begin{proposition}\label{z1z2}The following identities hold true:
\begin{enumerate}
\item $2K(\xi)\int\limits_{\D^2}e^V\mathcal Z_1^2=2K(\xi)\int\limits_{\D^2}e^V\mathcal Z_2^2=\pi\frac{2K(\xi)\(3\varphi(\xi)^2+K(\xi)\)}{3\varphi(\xi)^2\(\varphi(\xi)^2+K(\xi)\)^3}$.
\item $\frac{h(\xi)2\varphi(\xi)}{\varphi(\xi)^2+K(\xi)}\int\limits_{\partial\D^2}\mathcal Z_1^2=\frac{h(\xi)2\varphi(\xi)}{\varphi(\xi)^2+K(\xi)}\int\limits_{\partial\D^2}\mathcal Z_2^2=\pi\frac{\varphi(\xi)^2-K(\xi)}{\(\varphi(\xi)^2+K(\xi)\)^3}$.
\end{enumerate}
\end{proposition}

\begin{proof}
First of all, we notice that, since $\mathcal Z_2(z)=\mathcal Z_1(-\imath z)$, then the two interior integrals will be the same and also the two boundary integrals, so we will only prove them for $\mathcal Z_1$.\\
Integrating in polar coordinates we get:
\begin{eqnarray*}
&&2K(\xi)\int\limits_{\D^2}e^V\mathcal Z_1^2\\
&=&8K(\xi)\varphi(\xi)\int_0^1\(\int_{-\pi}^\pi\frac{r^2\cos^2t}{\(\varphi(\xi)^2+K(\xi)r^2\)^4}rdt\)dr\\
&=&8K(\xi)\varphi(\xi)\pi\int_0^1\frac{r^3}{\(\varphi(\xi)^2+K(\xi)r^2\)^4}dr\\
&=&8K(\xi)\varphi(\xi)\pi\int_0^1\(\frac1{K(\xi)}\frac r{\(\varphi(\xi)^2+K(\xi)r^2\)^3}-\frac{\varphi(\xi)^2}{K(\xi)}\frac r{\(\varphi(\xi)^2+K(\xi)r^2\)^4}\)dr\\
&=&8K(\xi)\varphi(\xi)\pi\(\frac1{K(\xi)}\frac{2\varphi(\xi)^2+K(\xi)}{4\varphi(\xi)^4\(\varphi(\xi)^2+K(\xi)\)^2}-\frac{\varphi(\xi)^2}{K(\xi)}\frac{3\varphi(\xi)^4+3\varphi(\xi)^2K(\xi)+K(\xi)^2}{6\varphi(\xi)^6\(\varphi(\xi)^2+K(\xi)\)^3}\)\\
&=&\pi\frac{2K(\xi)\(3\varphi(\xi)^2+K(\xi)\)}{3\varphi(\xi)^2\(\varphi(\xi)^2+K(\xi)\)^3}.
\end{eqnarray*}
The second identity follows from
\begin{eqnarray*}
\frac{h(\xi)2\varphi(\xi)}{\varphi(\xi)^2+K(\xi)}\int\limits_{\partial\D^2}\mathcal Z_1^2
&=&\frac{\varphi(\xi)^2-K(\xi)}{\(\varphi(\xi)^2+K(\xi)\)^3}\int_{-\pi}^\pi\cos^2tdt=\pi\frac{\varphi(\xi)^2-K(\xi)}{\(\varphi(\xi)^2+K(\xi)\)^3}.
\end{eqnarray*}
\end{proof}

\section*{Acknowledgments}

The authors are extremely grateful to David Ruiz for suggesting the problem, for reading the manuscript and for many useful conversations and suggestions.

\end{document}